\documentclass[11pt]{article}%
\usepackage[a4paper,margin=1in]{geometry}%
\usepackage[utf8]{inputenc}
\usepackage{amsmath,amssymb,amsfonts,amsthm,mathtools}
\usepackage{bm,mathrsfs,booktabs,array,enumitem,microtype}
\usepackage{graphicx,xcolor}
\usepackage[hidelinks]{hyperref}
\allowdisplaybreaks

\newtheorem{theorem}{Theorem}[section]
\newtheorem{proposition}[theorem]{Proposition}
\newtheorem{lemma}[theorem]{Lemma}
\newtheorem{corollary}[theorem]{Corollary}
\newtheorem{assumption}[theorem]{Assumption}
\theoremstyle{definition}
\newtheorem{definition}[theorem]{Definition}
\newtheorem{remark}[theorem]{Remark}

\newcommand{\RR}{\mathbb{R}}
\newcommand{\LT}{\mathscr{L}}
\newcommand{\FT}{\mathscr{F}}
\newcommand{\Xspace}{L^1(\RR^n)\cap L^\infty(\RR^n)}
\newcommand{\norm}[1]{\left\lVert #1\right\rVert}
\newcommand{\ip}[2]{\left\langle #1,#2\right\rangle}
\newcommand{\e}{\mathrm{e}}

\newcommand{\Aa}{\mathcal A}
\newcommand{\Nn}{\mathcal N}
\newcommand{\Log}{\operatorname{Log}}

\title{Well-posedness and Blow-up in a semilinear heat equation with variable-order Scarpi memory}
\author{Pu Yuan\thanks{Department of Mathematics, Utrecht University, The Netherlands.}
\and P. A. Zegeling\thanks{Department of Mathematics, Utrecht University, The Netherlands.}}
\date{}

\begin{document}
\maketitle

\begin{abstract}
We study the semilinear heat equation ${}^{S}D_0^{\alpha(t)}u=\Delta u+u^p$ on $\RR^n$ with $p>1$, where the variable Scarpi order relaxes exponentially at rate $c>0$ from $\alpha_1$ to $\alpha_2$ ($0<\alpha_1,\alpha_2<1$). 
For abstract Scarpi–Volterra equations, we establish local well-posedness and the existence of unique maximal mild solutions. For the whole-space heat equation, positivity of the resolvent families yields nonnegativity, comparison, an identity for the spatial integral of the solution, and an \(L^\infty\) blow-up alternative.

Using a Gaussian version of Kaplan's weighted-moment method, we reduce the PDE to a scalar nonlinear Volterra inequality without requiring pointwise lower estimates for the non-self-similar Scarpi heat kernel.
Under Bernstein admissibility, every nontrivial nonnegative solution blows up in finite time for $1<p<1+2/n$.
For every $p>1$ and every fixed nontrivial nonnegative $\varphi\in L^1\cap L^\infty$, the solution with $u_0=A\varphi$ blows up in finite time for sufficiently large $A$, and its maximal lifespan satisfies $T_A \asymp A^{-(p-1)/\alpha_1}$ as $A\to\infty$. 

More generally, lifespan bounds expressed via the inverse of the integrated memory capture the transition from $\alpha_1$ at large amplitudes to $\alpha_2$ at small amplitudes. 
We provide an explicit sufficient Bernstein criterion in the appendix and, combining it with continuous Laplace inversion and comparisons with both Caputo endpoint models, show that the transition regime produces nonmonotonic variations relative to the endpoint dynamics and an amplitude-dependent reversal in how the transition rate $c$ shifts prescribed amplitude-growth times.

\end{abstract}

\section{Introduction}
\label{sec:introduction}

Scarpi's variable-order fractional calculus provides a
convolution-based framework for describing temporal
memory with different short- and long-time behaviours
\cite{Scarpi1972,GarrappaGiusti2023}.
This framework has been investigated in models of
linear viscoelasticity, including variable-order
Maxwell models of stress relaxation and creep
\cite{giusti2024variable}, and in renewal processes
and associated continuous-time random walks
\cite{BeghinCristofaroGarrappa2024}.
Motivated by these developments, we investigate how
a transition in temporal memory influences nonlinear
growth and finite-time blow-up in the Cauchy problem
\begin{equation}
\label{eq:model}
\begin{cases}
 {}^{S}D_0^{\alpha(t)}u(x,t)
 =\Delta u(x,t)+u(x,t)^p,
 &x\in\RR^n,\quad t>0,\\
 u(x,0)=u_0(x)\ge0,
 &x\in\RR^n,
\end{cases}
\end{equation}
where $p>1$, $u_0\in L^1(\RR^n)\cap L^\infty(\RR^n)$,
and the fractional order follows the exponential
transition
\begin{equation}
\label{eq:transition}
\alpha(t)=\alpha_2+(\alpha_1-\alpha_2)\e^{-ct},
\qquad 0<\alpha_1,\alpha_2<1,\quad c>0.
\end{equation}
The Scarpi derivative is defined by a convolution
kernel determined through the Laplace transform
of $\alpha(t)$, rather than by directly substituting
a time-dependent order into the Caputo kernel.
For \eqref{eq:transition}, the Laplace-domain multiplier associated with the derivative is
\[
 \Psi(z)^{-1}=z^{(\alpha_2c+\alpha_1z)/(c+z)}.
\]
Its high- and low-frequency limits correspond to the Caputo orders $\alpha_1$ and $\alpha_2$.
We refer to these as the initial order and the long-time order.

For the classical semilinear heat equation on $\RR^n$, every nontrivial nonnegative solution blows up in finite time when $1<p\le p_F=1+2/n$, whereas suitably small nonnegative
initial data generate global solutions when $p>p_F$ \cite{Fujita1966,KobayashiSiraoTanaka1977}. 
For the corresponding equation with a constant-order Caputo time derivative, the Fujita exponent remains $p_F$. At $p=p_F$, both global solutions and solutions that blow up in finite time exist, depending on initial conditions \cite{ZhangSun2015,cortazar2024semilinear}.

To establish such blow-up results, Kaplan's weighted-moment method provides a classical approach to studying the competition between diffusion and nonlinear growth \cite{Kaplan1963}. Related blow-up results have been established for equations with general convolution kernels \cite{VergaraZacher2017}, for Caputo equations on bounded domains \cite{FloridiaLiuYamamoto2023,HuangLiuYamamoto2024}, and generalised fractional derivatives such as Caputo--Hadamard or $\psi$-Caputo models \cite{li2021blow,LiLi2022PsiCaputo}. 
By contrast, the variable-order Scarpi operator introduces a memory structure with distinct short- and long-time power-law regimes, bringing two analytical difficulties.
Its memory kernel is not self-similar, so constant-order scaling formulae do not capture the intermediate transition phase; moreover, the endpoint powers alone do not supply the positivity needed for comparison and weighted-moment arguments.
Consequently, the nonlinear problem requires a solution framework that separates general well-posedness from the structural assumptions needed for blow-up analysis.
Within that framework, the central questions are whether the subcritical Fujita blow-up range persists, which memory order governs the lifespan scales, and how the transition shapes nonlinear growth and spatial spreading between those scales.

We first establish local well-posedness and the existence of unique maximal mild solutions for abstract Scarpi--Volterra equations with sectorial linear operators satisfying \eqref{eq:sectorial-resolvent} and nonlinearities that are Lipschitz on bounded sets.
The construction gives continuous dependence and a norm blow-up alternative, with continuation retaining the full memory history.
For the heat equation, we then impose the Bernstein-admissibility condition in Assumption~\ref{ass:admissible}.
Subordination gives positive convolution kernels $Z(t)$ and $Y(t)$ with $Z(t,\mathbb R^n)=1$ and $Y(t,\mathbb R^n)=\psi(t)$ for almost every $t>0$.
This structure gives nonnegative mild solutions for initial values in $L^1\cap L^\infty$, a comparison principle, a mass identity, and an $L^\infty$ continuation criterion.

The blow-up argument combines a normalised Gaussian test function with a scalar Volterra iteration.
The Gaussian controls the diffusion term and closes the reaction term by Jensen's inequality.
The iteration uses a lower weak scaling bound for
\[
 K(t)=\int_0^t\psi(s)\,ds,
 \qquad \widehat\psi=\Psi,
\]
so no pointwise lower estimate for the Scarpi heat kernel is needed.
Under Bernstein admissibility, every nontrivial nonnegative solution blows up in finite time for $1<p<1+2/n$, and sufficiently large multiples of any nonzero nonnegative initial profile blow up for every $p>1$.
At $p=1+2/n$, the condition $|u_0|_1>(4\pi n)^{n/2}$ is sufficient for finite-time blow-up.
The subcritical range comes from the spatial Gaussian balance, whereas $K$ determines the time scale in the resulting estimates.

The local-existence lower bound and Gaussian upper bound are expressed through $K^{-1}$.
For a fixed nonzero $0\le\varphi\in L^1\cap L^\infty$, let $T_A$ denote the maximal lifespan for $u_0=A\varphi$.
For fixed admissible memory parameters, we obtain
\[
 T_A\asymp A^{-(p-1)/\alpha_1}
 \quad\text{ as }A\to\infty.
\]
For $u_0=A_\varepsilon\varphi$ in the subcritical range, we also prove
\[
 T_\varepsilon\lesssim
 A_\varepsilon^{-\frac{1}{\alpha_2(1/(p-1)-n/2)}}
 \quad\text{ as }A_\varepsilon\to0^+.
\]
Thus the initial order determines the matching large-amplitude lifespan scale, while the long-time order enters the small-amplitude upper bound.
The inverse-$K$ estimates retain the memory transition between the two asymptotic limits.

The numerical study examines how the memory transition affects nonlinear growth and spatial spreading.
We combine Fourier spatial approximation with a fully implicit backward-Euler scheme from Lubich's convolution quadrature framework \cite{Lubich1986,Lubich1988I,Lubich1988II}.
The computed Scarpi solutions are compared with both constant-order Caputo models.
Continuous Laplace inversion shows that the
logarithmic slope of $K$ can lie outside the interval bounded by $\alpha_1$ and $\alpha_2$, despite the monotonicity of $\alpha(t)$.
At large amplitudes, fixed-profile threshold times and rescaled solutions become closer to those of Caputo $\alpha_1$ over the computed range.

Fixed-width experiments distinguish these rescaling effects from fixed-profile amplitude variation.
They show that increasing the transition rate can delay a fixed relative growth threshold at larger amplitudes and advance it at smaller amplitudes.
Profiles compared at equal peak height also reveal differences in spatial spreading.
These observations complement the theoretical lifespan scales by providing an amplitude-dependent description of the memory transition.
The numerical parameters are certified by the sufficient conditions in Appendix~\ref{app:bernstein-b}, with representative exact certificates reported in Table~\ref{tab:BF-b-certificates}.

The remainder of the paper is organised as follows.
Section~\ref{sec:scarpi} constructs the Scarpi kernels and resolvent families and establishes the integrated-memory estimates.
Section~\ref{sec:wellposedness} develops the abstract and whole-space solution theory together with the weak test identity.
Sections~\ref{sec:blowup} and~\ref{sec:lifespan} prove the blow-up criteria and lifespan bounds.
Section~\ref{sec:numerics} presents the numerical method and the Scarpi--Caputo comparisons, and Appendix~\ref{app:bernstein-b} establishes the sufficient conditions for Bernstein admissibility.

\section{Scarpi memory and Volterra resolvents}
\label{sec:scarpi}

\subsection{Sonine kernels and resolvent families}
\label{subsec:sonine-symbol}

We recall the framework of variable-order fractional calculus adopted in this work, based on the Laplace-domain formulation of Scarpi-type operators \cite{Scarpi1972,GarrappaGiusti2023,GarrappaGiustiMainardi2021}.
\begin{definition}
Let $\alpha\in(0,1)$ be a constant.
The Riemann--Liouville fractional integral and the Caputo fractional derivative are defined by
\[\begin{aligned}
(I_0^{\alpha} f)(t)
&= \int_0^t \psi_\alpha(t-\tau) f(\tau)\,d\tau,
\qquad
\psi_\alpha(t)=\frac{t^{\alpha-1}}{\Gamma(\alpha)}, \\
(D_0^{\alpha} f)(t)
&= \int_0^t \phi_\alpha(t-\tau) f'(\tau)\,d\tau,
\qquad
\phi_\alpha(t)=\frac{t^{-\alpha}}{\Gamma(1-\alpha)}.
\end{aligned}\]
\end{definition}
The Laplace transforms of the kernels are
\[
\LT\{\psi_\alpha\}(z)=z^{-\alpha},
\qquad
\LT\{\phi_\alpha\}(z)=z^{\alpha-1},
\]
and hence $\LT\{\phi_\alpha*\psi_\alpha\}=z^{-1}$, where \((\phi_\alpha*\psi_\alpha)(t)=\int_0^t \phi_\alpha(t-s)\psi_\alpha(s)\,ds\).
Therefore, $(\psi_\alpha,\phi_\alpha)$ form a Sonine pair, and this provides the convolutional inverse structure underlying the fractional integral and derivative.
Consequently,
\[
D_0^{\alpha} I_0^{\alpha} f = f,
\qquad
I_0^{\alpha} D_0^{\alpha} f = f - f(0),
\]
under suitable regularity assumptions \cite{Diethelm2010,SamkoKilbasMarichev1993}.

Based on the Sonine relation, we now introduce the variable-order framework.
Following \cite{GarrappaGiusti2023}, Scarpi's construction may be formulated for a locally integrable function \(\alpha:[0,\infty)\to[0,1]\) whose Laplace transform is known in a right half-plane.
In this work we restrict throughout to the continuous exponential transition \eqref{eq:transition}.

\begin{definition}
Let $\alpha(t)=\alpha_2+(\alpha_1-\alpha_2)e^{-ct}$ with
$0<\alpha_1,\alpha_2<1$ and $c>0$.
Its Laplace transform and order symbol are
\[
\widehat\alpha(z)=\LT\{\alpha\}(z)
=\frac{\alpha_2c+\alpha_1z}{z(c+z)},
\qquad
B(z)=z\widehat\alpha(z)=\frac{\alpha_2c+\alpha_1z}{c+z},
\qquad \Re z>0.
\]

Using the principal branch of $\Log z$ on the chosen Laplace
inversion domain, define the associated kernel symbols by
\[
\Psi(z)=z^{-B(z)},
\qquad
\Phi(z)=z^{B(z)-1}=\frac{1}{z\Psi(z)}.
\]
The corresponding time-domain kernels are defined by the inverse Laplace transforms:
\begin{equation}
\label{eq:inverse_LT_kernel}
\psi(t)=\LT^{-1}\{\Psi\}(t),
\qquad
\phi(t)=\LT^{-1}\{\Phi\}(t).
\end{equation}

\end{definition}
We take $\Log z=\log|z|+i\operatorname{Arg}z$, where $\operatorname{Arg}z\in(-\pi,\pi)$, so that \(B\), \(\Psi\), \(\Phi\), and \(z\mapsto\Psi(z)^{-1}\) are analytic on \(\mathbb C\setminus(-\infty,0]\).
The symbols satisfy
\[
\Psi(z)\Phi(z)=\frac1z.
\]
The explicit symbol has two power laws:
\begin{equation}
\label{eq:symbol-asymp}
\Psi(z)^{-1}\sim z^{\alpha_1}\quad(z\to\infty),
\qquad
\Psi(z)^{-1}\sim z^{\alpha_2}\quad(z\to0^+).
\end{equation}
Indeed,
\begin{equation}
\label{eq:B-difference}
B(z)-\alpha_1=\frac{c(\alpha_2-\alpha_1)}{z+c},
\qquad
B(z)-\alpha_2=\frac{(\alpha_1-\alpha_2)z}{z+c},
\end{equation}
so $(B(s)-\alpha_j)\log s\to0$ in the corresponding limit.
The high-frequency asymptotic behaviour controls the singularity of the forcing family $P_{\mathcal{A}}(t)$ at the origin, whereas the low-frequency asymptotic behaviour governs the long-time integrated memory used in the lifespan analysis.

We first construct the linear resolvent families independently of positivity.
Let $E$ be a complex Banach space and let $\Aa:D(\Aa)\subset E\to E$ be densely defined and closed.
Assume that $\Aa$ is nonpositive sectorial in the following sense: for every $\eta\in(0,\pi)$ there is $M_\eta\ge1$ such that
\begin{equation}
\label{eq:sectorial-resolvent}
\Sigma_{\pi-\eta}\subset\rho(\Aa),
\qquad
\norm{(\lambda I-\Aa)^{-1}}_{\mathcal L(E)}
\le M_\eta|\lambda|^{-1},
\qquad \lambda\in\Sigma_{\pi-\eta},
\end{equation}
where $\Sigma_\vartheta=\{\lambda\ne0:|\operatorname{Arg}\lambda|<\vartheta\}$.
The Laplacian on $L^q(\RR^n)$ for $1 \le q < \infty$, and on $\mathrm{BUC}(\RR^n)$, equipped with its standard generator domain, fits into this setting \cite{Pazy1983,Pruss1993}.

Fix $\theta\in(\pi/2,\pi)$ for
\[
 \eta_\theta=\frac{\pi-\alpha_1\theta}{3}>0.
\]
For $\kappa>0$, let $\Gamma_{\theta,\kappa}$ be the oriented contour
\[
 \{re^{-i\theta}:\infty>r\ge\kappa\}
 \cup\{\kappa e^{i\chi}:-\theta\le\chi\le\theta\}
 \cup\{re^{i\theta}:\kappa\le r<\infty\}.
\]

\begin{lemma}
\label{lem:scarpi-sector}
There is $\kappa_0\ge\max\{1,2c\}$ such that, whenever $z=re^{i\chi}$, $r\ge\kappa_0$, and $|\chi|\le\theta$,
\begin{equation}
\label{eq:exterior-Psi-inverse}
|\operatorname{Arg}\Psi(z)^{-1}|\le\pi-2\eta_\theta,
\qquad
e^{-\eta_\theta}r^{\alpha_1}
\le |\Psi(z)^{-1}|\le
e^{\eta_\theta}r^{\alpha_1}.
\end{equation}
Consequently, $\Psi(z)^{-1}\in\rho(\Aa)$ and
\begin{equation}
\label{eq:exterior-resolvent}
\norm{(\Psi(z)^{-1}I-\Aa)^{-1}}_{\mathcal L(E)}
\le C_\theta |\Psi(z)|
\le C_\theta r^{-\alpha_1}.
\end{equation}
\end{lemma}

\begin{proof}
By \eqref{eq:B-difference}, for $r\ge\max\{1,2c\}$,
\[
 \bigl|(B(z)-\alpha_1)\Log z\bigr|
 \le \frac{2c|\alpha_2-\alpha_1|}{r}(\log r+\theta),
\]
uniformly for $|\chi|\le\theta$.
Choose $\kappa_0$ so that the right-hand side is at most $\eta_\theta$.
Since
\[
 B(z)\Log z=\alpha_1(\log r+i\chi)
 +(B(z)-\alpha_1)\Log z,
\]
we obtain
\[
 |\Im(B(z)\Log z)|
 \le\alpha_1\theta+\eta_\theta=\pi-2\eta_\theta<\pi
\]
and $|\Re(B(z)\Log z)-\alpha_1\log r|\le\eta_\theta$.
There is therefore no $2\pi$ wrapping of the principal argument, and \eqref{eq:exterior-Psi-inverse} follows.
Since $\Sigma_{\pi-2\eta_\theta}\subset\Sigma_{\pi-\eta_\theta}$, \eqref{eq:sectorial-resolvent}, applied to $\lambda=\Psi(z)^{-1}$, gives \eqref{eq:exterior-resolvent}.
\end{proof}

\begin{proposition}
\label{prop:scarpi-kernels}
Fix \(\theta\in(\pi/2,\pi)\), and let \(\kappa_0\) be given by Lemma~\ref{lem:scarpi-sector}.
For \(t>0\), define
\begin{equation}
\label{eq:scalar-scarpi-kernels}
\psi(t)
=
\frac{1}{2\pi i}
\int_{\Gamma_{\theta,\kappa_0}}
e^{zt}\Psi(z)d z,
\qquad
\phi(t)
=
\frac{1}{2\pi i}
\int_{\Gamma_{\theta,\kappa_0}}
e^{zt}\Phi(z)d z .
\end{equation}
The contour integrals converge absolutely and define real-valued functions of exponential order.
For every \(T>0\), there is a constant \(C_T>0\) such that
\begin{equation}
\label{eq:scarpi-kernel-local-bounds}
|\psi(t)|
\le C_T t^{\alpha_1-1},
\qquad
|\phi(t)|
\le C_T t^{-\alpha_1},
\qquad
0<t\le T.
\end{equation}
In particular, $\psi,\phi\in L^1_{\mathrm{loc}}(0,\infty).$ For \(\Re z\) sufficiently large,
\begin{equation*}
\widehat\psi(z)=\Psi(z),
\qquad
\widehat\phi(z)=\Phi(z)=\frac{1}{z\Psi(z)} .
\end{equation*}
Moreover, \(\psi\) and \(\phi\) form a Sonine pair:
\begin{equation*}
(\psi*\phi)(t)=1
\quad\text{for a.e. }t>0.
\end{equation*}
\end{proposition}
\begin{proof}
The integrands in \eqref{eq:scalar-scarpi-kernels} are analytic in every exterior annular sector bounded by the admissible contours. By Cauchy's theorem, we may therefore replace $\kappa_0$ with any $\kappa\ge\kappa_0$.

Let $\kappa(t)=\max\{\kappa_0,t^{-1}\}.$ By Lemma~\ref{lem:scarpi-sector},
\[
|\Psi(z)|\le C|z|^{-\alpha_1},
\qquad
|\Phi(z)|=\frac{|\Psi(z)^{-1}|}{|z|}
\le C|z|^{\alpha_1-1}
\]
on \(\Gamma_{\theta,\kappa(t)}\). Let $q_\theta=-\cos\theta>0$. On either ray, the integrals satisfy
\[
\int_{\kappa(t)}^\infty
e^{-q_\theta rt}r^{-\alpha_1}d r
\le Ct^{\alpha_1-1}\ \text{and}\ 
\int_{\kappa(t)}^\infty
e^{-q_\theta rt}r^{\alpha_1-1}d r
\le Ct^{-\alpha_1}.
\]
On the circular arc, the corresponding contributions are bounded by $C\kappa(t)^{1-\alpha_1}e^{\kappa(t)t}$ and $C\kappa(t)^{\alpha_1}e^{\kappa(t)t}$. 
Since
\[
1\le \kappa(t)t\le\max\{1,\kappa_0T\},
\qquad 0<t\le T,
\]
these terms are bounded by \(C_Tt^{\alpha_1-1}\) and \(C_Tt^{-\alpha_1}\), respectively.
This proves \eqref{eq:scarpi-kernel-local-bounds} and hence the local integrability of both kernels.
The same fixed-contour estimates for \(t\ge1\) show that the kernels are of exponential order.

For \(\Re z\) sufficiently large, Fubini's theorem and the sectorial Cauchy formula give
\[
\begin{aligned}
\int_0^\infty e^{-zt}\psi(t)d t
&=
\frac{1}{2\pi i}
\int_{\Gamma_{\theta,\kappa_0}}
\frac{\Psi(s)}{z-s}d s
=
\Psi(z),\\
\int_0^\infty e^{-zt}\phi(t)d t
&=
\frac{1}{2\pi i}
\int_{\Gamma_{\theta,\kappa_0}}
\frac{\Phi(s)}{z-s}d s
=
\Phi(z).
\end{aligned}
\]
The Cauchy integrands on the rays are \(O(|s|^{-\alpha_1-1})\) and \(O(|s|^{\alpha_1-2})\), respectively, and both are integrable because \(0<\alpha_1<1\).

Finally,
\[
\widehat{(\psi*\phi)}(z)
=
\widehat\psi(z)\widehat\phi(z)
=
\Psi(z)\Phi(z)
=
\frac1z.
\]
Laplace-transform uniqueness yields \((\psi*\phi)(t)=1\) for a.e. \(t>0\).
The conjugation identities
\[
\overline{\Psi(\overline z)}=\Psi(z),
\qquad
\overline{\Phi(\overline z)}=\Phi(z)
\]
and the symmetry of the contour show that both kernels are real-valued.
\end{proof}

Let \(E\) be a Banach space and \(T>0\).
For \(g\in L^1(0,T;E)\) and \(u\in W^{1,1}(0,T;E)\), define
\begin{equation*}
{}^{S}I_0^{\alpha(t)}g
=
\psi*g,
\qquad
{}^{S}D_0^{\alpha(t)}u
=
\phi*u'.
\end{equation*}

\begin{corollary}
\label{cor:scarpi-inverse-identities}
Let \(E\) be a Banach space and \(T>0\).
For every \(u\in W^{1,1}(0,T;E)\), we have
\begin{equation*}
{}^{S}I_0^{\alpha(t)}
{}^{S}D_0^{\alpha(t)}u
=
u-u(0)
\quad\text{a.e. on }(0,T).
\end{equation*}
Let $W^{1,1}_0(0,T;E)=\{v\in W^{1,1}(0,T;E):v(0)=0\}$.
If $g\in L^1(0,T;E)$ and
$\psi*g\in W^{1,1}_0(0,T;E)$, then
\[
{}^{S}D_0^{\alpha(t)}
{}^{S}I_0^{\alpha(t)}g
=
g
\quad\text{a.e. on }(0,T).
\]
\end{corollary}

\begin{proof}
All convolutions below are well defined on \((0,T)\), and their associativity follows from the local \(L^1\) bounds in Proposition~\ref{prop:scarpi-kernels}.
For \(u\in W^{1,1}(0,T;E)\),
\[
{}^{S}I_0^{\alpha(t)}
{}^{S}D_0^{\alpha(t)}u
=
\psi*(\phi*u')
=
(\psi*\phi)*u'
=u-u(0).
\]

Now let \(v=\psi*g\in W^{1,1}_0(0,T;E)\).
Since \(v(0)=0\), the standard differentiation rule for Volterra convolutions gives
\[
\phi*v'
=
\frac{d}{dt}(\phi*v).
\]
Associativity and the Sonine identity yield
\[
{}^{S}D_0^{\alpha(t)}
{}^{S}I_0^{\alpha(t)}g
=
\phi*v'
=
\frac{d}{dt}(1*g)
=
g
\]
almost everywhere on \((0,T)\).
\end{proof}
Corollary \ref{cor:scarpi-inverse-identities} then gives the following equivalence between the Scarpi differential equation and the corresponding Volterra integral equation.

\begin{corollary}
\label{cor:strong-volterra-equivalence}
Let $E$ be a Banach space and let $\Nn:E\to E$ be a nonlinear map.
Suppose that $u\in W^{1,1}(0,T;D(\Aa))$, where $D(\Aa)$ is equipped with its graph norm, and that
$\Aa u+\Nn(u)\in L^1(0,T;E)$.
Then $u$ is a strong solution of 
\[
{}^{S}D_0^{\alpha(t)}u
=
\Aa u+\Nn(u)
\quad\text{a.e. on }(0,T)
\]
if and only if
\begin{equation}
\label{eq:abstract-volterra-form}
u(t) = u(0)+ \int_0^t \psi(t-s)\bigl[\Aa u(s)+\Nn(u(s))\bigr]\,ds \quad\text{for a.e. }t\in(0,T).
\end{equation}
\end{corollary}
The preceding result identifies the strong and Volterra formulations under explicit time and domain regularity.
The mild formulation is constructed below from the operator families \(S_{\Aa}\) and \(P_{\Aa}\).
For $t>0$, define
\begin{align}
\notag
P_{\Aa}(t)
&=\frac{1}{2\pi i}\int_{\Gamma_{\theta,\kappa_0}}
 e^{zt}(\Psi(z)^{-1}I-\Aa)^{-1}d z,\\
\label{eq:S-family}
S_{\Aa}(t)
&=\frac{1}{2\pi i}\int_{\Gamma_{\theta,\kappa_0}}
 e^{zt}\frac{1}{z\Psi(z)}(\Psi(z)^{-1}I-\Aa)^{-1}d z.
\end{align}
The integrands are analytic in every exterior annular sector on which \eqref{eq:exterior-resolvent} holds.
Cauchy's theorem therefore makes the values independent of replacing $\kappa_0$ by any $\kappa\ge\kappa_0$.

\begin{proposition}
\label{prop:linear-families}
For every $T>0$,
\begin{equation}
\label{eq:P-regularity}
P_{\Aa}\in C((0,\infty);\mathcal L(E))
\cap L^1((0,T);\mathcal L(E)),
\qquad
\norm{P_{\Aa}(t)}_{\mathcal L(E)}\le C_Tt^{\alpha_1-1}.
\end{equation}
The family $S_{\Aa}$ is norm-continuous for $t>0$, locally uniformly bounded, and has a unique strongly continuous extension to $t=0$ with $S_{\Aa}(0)=I$.
Moreover, for $\Re z$ sufficiently large,
\begin{equation}
\label{eq:PS-Laplace}
\widehat{P_{\Aa}}(z)=(\Psi(z)^{-1}I-\Aa)^{-1},
\qquad
\widehat{S_{\Aa}}(z)=\frac{1}{z\Psi(z)}
(\Psi(z)^{-1}I-\Aa)^{-1}.
\end{equation}
If $x\in D(\Aa)$, then
\begin{equation}
\label{eq:domain-invariance}
S_{\Aa}(t)x\in D(\Aa),
\qquad
\Aa S_{\Aa}(t)x=S_{\Aa}(t)\Aa x,
\end{equation}
and
\begin{align}
\label{eq:S-P-identity}
S_{\Aa}(t)x
&=x+\int_0^tP_{\Aa}(r)\Aa xd r,\\
\label{eq:S-volterra-identity}
S_{\Aa}(t)x
&=x+\int_0^t\psi(t-r)\Aa S_{\Aa}(r)xd r.
\end{align}
\end{proposition}

\begin{proof}
Let $q_\theta=-\cos\theta>0$ and $\kappa(t)=\max\{\kappa_0,t^{-1}\}$.
Contour independence permits the use of $\Gamma_{\theta,\kappa(t)}$.
On either ray, \eqref{eq:exterior-resolvent} gives
\[
 \int_{\kappa(t)}^\infty e^{-q_\theta rt}r^{-\alpha_1}d r
 \le Ct^{\alpha_1-1}.
\]
On the circular arc, the contribution to $P_{\Aa}(t)$ is bounded by $C\kappa(t)^{1-\alpha_1}e^{\kappa(t)t}$.
Since $1\le\kappa(t)t\le\max\{1,\kappa_0T\}$ for $0<t\le T$, this is also bounded by $C_Tt^{\alpha_1-1}$.
This proves the pointwise estimate in \eqref{eq:P-regularity}, and its integrability follows from $0<\alpha_1<1$.
On a compact subinterval of $(0,\infty)$, the fixed-contour integrand has an integrable time-independent majorant.
Dominated convergence therefore gives norm continuity of $P_{\Aa}$.

For $S_{\Aa}$, the sectorial estimate yields
\[
 \norm{\Psi(z)^{-1}(\Psi(z)^{-1}I-\Aa)^{-1}}\le C_\theta.
\]
The ray and circular contributions are bounded, respectively, by
\[
 C\int_{\kappa(t)}^\infty e^{-q_\theta rt}\frac{d r}{r}
 \quad\text{and}\quad
 C\int_{-\theta}^{\theta}
 e^{\kappa(t)t\cos\chi}d\chi.
\]
Both are uniformly bounded on $0<t\le T$.
The fixed-contour argument again gives norm continuity for $t>0$.

Let $x\in D(\Aa)$.
Resolvent commutation gives
\[
 \Aa(\Psi(z)^{-1}I-\Aa)^{-1}x
 =(\Psi(z)^{-1}I-\Aa)^{-1}\Aa x.
\]
On either ray, the graph-norm integrand satisfies
\[
\begin{aligned}
&e^{t\Re z}\left(
 \norm{\frac{1}{z\Psi(z)}(\Psi(z)^{-1}I-\Aa)^{-1}x}
 +\norm{\Aa\frac{1}{z\Psi(z)}(\Psi(z)^{-1}I-\Aa)^{-1}x}
 \right)\\
&\hspace{35mm}\le
 Ce^{-q_\theta rt}r^{-1}
 \bigl(\norm{x}+\norm{\Aa x}\bigr),
\end{aligned}
\]
and the corresponding integrand is bounded on the circular arc.
Thus \eqref{eq:S-family} converges as a Bochner integral in $D(\Aa)$ equipped with its graph norm.
This proves \eqref{eq:domain-invariance}.

For $\Re z$ sufficiently large, Fubini's theorem and the sectorial Cauchy formula yield
\[
\begin{aligned}
\int_0^\infty e^{-zt}P_{\Aa}(t)d t
&=\frac{1}{2\pi i}\int_{\Gamma_{\theta,\kappa_0}}
 \frac{(\Psi(s)^{-1}I-\Aa)^{-1}}{z-s}d s
 =(\Psi(z)^{-1}I-\Aa)^{-1},\\
\int_0^\infty e^{-zt}S_{\Aa}(t)d t
&=\frac{1}{2\pi i}\int_{\Gamma_{\theta,\kappa_0}}
 \frac{(\Psi(s)^{-1}I-\Aa)^{-1}}{s\Psi(s)(z-s)}d s
 =\frac{1}{z\Psi(z)}(\Psi(z)^{-1}I-\Aa)^{-1}.
\end{aligned}
\]
On the rays, the two Cauchy integrands are respectively $O(|s|^{-\alpha_1-1})$ and $O(|s|^{-2})$, which justifies the interchange and the limiting Cauchy integrals.
This proves \eqref{eq:PS-Laplace}.

The resolvent identity
\[
 \Psi(z)^{-1}(\Psi(z)^{-1}I-\Aa)^{-1}x
 =x+(\Psi(z)^{-1}I-\Aa)^{-1}\Aa x
\]
now gives \eqref{eq:S-P-identity} by Laplace uniqueness.
In particular,
\[
 \norm{S_{\Aa}(t)x-x}
 \le\norm{\Aa x}\int_0^t\norm{P_{\Aa}(r)}d r\longrightarrow0
 \qquad(t\to 0^+)
\]
for $x\in D(\Aa)$.
If $x_m\in D(\Aa)$ and $x_m\to x$ in $E$, then, with $M_T=\sup_{0<t\le T}\norm{S_{\Aa}(t)}$,
\[
 \norm{S_{\Aa}(t)x-x}
 \le(M_T+1)\norm{x-x_m}+\norm{S_{\Aa}(t)x_m-x_m}.
\]
First letting $t\to 0^+$ and then $m\to\infty$ proves strong continuity at zero for every $x\in E$ and fixes the only possible extension $S_{\Aa}(0)=I$.

Finally, since Proposition~\ref{prop:scarpi-kernels} gives \(\widehat\psi=\Psi\), using \eqref{eq:PS-Laplace} and resolvent commutation shows that the Laplace transform of the right-hand side of \eqref{eq:S-volterra-identity} equals
\[
 \frac{x}{z}+\Psi(z)\Aa\frac{1}{z\Psi(z)}
 (\Psi(z)^{-1}I-\Aa)^{-1}x
 =\frac{1}{z\Psi(z)}(\Psi(z)^{-1}I-\Aa)^{-1}x.
\]
Laplace uniqueness and continuity prove \eqref{eq:S-volterra-identity}.
\end{proof}

Through the \(S_{\Aa}\)--\(P_{\Aa}\) formula we can define solutions that need not take values in \(D(\Aa)\), which is the case for mild solutions.
A Volterra solution satisfies \eqref{eq:abstract-volterra-form}, so \(\Aa u+\Nn(u)\) must be defined in the underlying space or in distributions.
A strong solution additionally belongs to \(W^{1,1}\) and has a well-defined convolution \(\phi*u'\).
Corollary~\ref{cor:strong-volterra-equivalence} identifies the strong and Volterra formulations under these regularity assumptions.
The \(L^1\cap L^\infty\) mild solutions below are treated directly through the resolvent formula, and their weak Volterra identity is established in Proposition~\ref{prop:weak-identity}.

\subsection{Positive heat resolvent families}
\label{subsec:positive-heat}
In scalar Scarpi relaxation theory, a Bernstein condition on the symbol guarantees a nonnegative and nonincreasing relaxation law, and in the whole-space setting it supplies the positivity and exact spatial masses of the heat resolvent families needed for comparison and blow-up.
For the exponential transition \eqref{eq:transition}, a complete characterisation of the Bernstein-admissible parameter regime remains open, having been explored primarily via numerical Laplace inversion \cite{BeghinCristofaroGarrappa2024}.
We therefore impose the following Bernstein-admissibility condition.

\begin{assumption}
\label{ass:admissible}
For $z>0$, the function
\[
z\longmapsto\Psi(z)^{-1}
=z^{(\alpha_2c+\alpha_1z)/(c+z)}
\]
is a Bernstein function.
\end{assumption}

Theorem~\ref{thm:BF-b-sufficient} in Appendix~\ref{app:bernstein-b} provides a sufficient parameter criterion for Assumption~\ref{ass:admissible}.
In particular, for each fixed $\alpha_1\in(0,1)$ and $c>0$, the assumption holds whenever $|\alpha_2-\alpha_1|$ is sufficiently small.

\begin{proposition}
\label{prop:positive-pair}

Under Assumption~\ref{ass:admissible}, let $\Aa=\Delta$.
There are positive convolution operators $S_{\Aa}$ and $P_{\Aa}$, defined by
\begin{equation}
\label{eq:SP-convolution}
S_{\Aa}(t)g=Z(t)*g,
\qquad
P_{\Aa}(t)g=Y(t)*g,\quad\text{for a.e. }t>0,
\end{equation}
where  $Z(t)$ is a probability measure and $Y(t)$ is a finite positive measure on $\mathbb R^n$.
Their Fourier--Laplace transforms are
\begin{equation}
\label{eq:ZY-transforms}
(\LT_t\FT_x Z)(\xi,z)
=\frac{\Psi(z)^{-1}}{z(\Psi(z)^{-1}+|\xi|^2)},
\qquad
(\LT_t\FT_x Y)(\xi,z)
=\frac1{\Psi(z)^{-1}+|\xi|^2}.
\end{equation}
Furthermore, for a.e. $t>0$, these measures satisfy the mass identities
\begin{equation}
\label{eq:mass-kernels}
 Z(t,\RR^n)=1,
\qquad
 Y(t,\RR^n)=\psi(t),
\end{equation}
which yield the following bounds:
\begin{equation}
\label{eq:positive-resolvent-bounds}
\norm{S_{\Aa}(t)g}_q\le\norm{g}_q,
\qquad
\norm{P_{\Aa}(t)g}_q\le\psi(t)\norm{g}_q,\qquad q=1,\infty.
\end{equation}
In particular, $\psi(t)>0$ for a.e. $t>0$, and the map $t\mapsto S_{\Aa}(t)g$ is strongly continuous on $L^1$ and on $BUC$.
$S_{\Aa}$ is bounded on $L^\infty$, but not necessarily strongly continuous at zero.
\end{proposition}

\begin{proof}
By \eqref{eq:symbol-asymp}, $\Psi(z)^{-1}\to0$ as $z\to0^+$.
{Let $D_r$ be the subordinator with Laplace exponent $z\mapsto\Psi(z)^{-1}$, and let $\mu_r$ denote its law \cite{SchillingSongVondracek2012,Toaldo2015}.}
Its potential measure is
\[
 U(d t)=\int_0^\infty\mu_r(d t)d r.
\]
Tonelli's theorem gives, for $z>0$,
\[
\begin{aligned}
 \int_{[0,\infty)}e^{-zt}U(d t)
 &=\int_0^\infty\int_{[0,\infty)}e^{-zt}\mu_r(d t)d r\\
 &=\int_0^\infty e^{-r\Psi(z)^{-1}}d r
 =\Psi(z).
\end{aligned}
\]
It is locally finite because $U([0,T])\le e^{zT}\Psi(z)$ for every $T,z>0$.
Proposition~\ref{prop:scarpi-kernels} gives $\widehat\psi=\Psi$, with $\psi\in L^1_{\mathrm{loc}}(0,\infty)$ of exponential order.
Uniqueness of the Laplace transform for locally finite signed Radon measures with finite exponential moments therefore yields
\[
 U(d t)=\psi(t)d t.
\]
The contour estimate makes $\psi$ real analytic on $(0,\infty)$. Since $U(dt)=\psi(t)\,dt$ is positive, continuity gives $\psi(t)\geq0$ for $t>0$. Its Laplace transform is nonzero, so its zeros are isolated and $\psi(t)>0$ for a.e. $t>0$.

Let $G_r$ denote the Gaussian probability measure on $\RR^n$ for $r>0$ and define the positive space--time potential measure
\[
 \mathcal Y(d x,d t)
 =\int_0^\infty G_r(d x)\,\mu_r(d t)d r.
\]
Its time marginal is $\mathcal Y(\RR^n,d t)=U(d t)=\psi(t)d t$.
Moreover, $\mathcal Y(\RR^n\times[0,T])=U([0,T])<\infty$ for every $T>0$, so $\mathcal Y$ is $\sigma$-finite. Because $\RR^n$ and $[0,\infty)$ are standard Borel spaces, the disintegration theorem gives a measurable probability kernel $Q_t(d x)$ with respect to $U$, which we extend arbitrarily on the exceptional $U$-null set. 
Define $Y(t,d x)=\psi(t)Q_t(d x)$. Then
\[
 \mathcal Y(d x,d t)=Y(t,d x)d t,
 \qquad
 Y(t,\RR^n)=\psi(t)
\]
for a.e. $t>0$.

Let $E_t=\inf\{r>0:D_r>t\}$ and let $\eta_t$ be the law of $E_t$.
Define
\[
 Z(t,d x)=\int_0^\infty G_r(d x)\,\eta_t(d r).
\]
Then $Z(t,\cdot)$ is a probability measure on $\RR^n$ for every $t\ge 0$.
Define temporary positive convolution families by
\[
 \widetilde S(t)g=Z(t)*g,
 \qquad
 \widetilde P(t)g=Y(t)*g,
\]
where $\widetilde P$ is defined almost everywhere.
The identities
\[
 \int_{[0,\infty)}e^{-zt}\mu_r(d t)=e^{-r\Psi(z)^{-1}},
\qquad
\int_0^\infty e^{-zt}\eta_t(d r)d t
 =\frac{1}{z\Psi(z)}e^{-r\Psi(z)^{-1}}d r
\]
for $z>0$ give
\[
 \widehat{\widetilde P}(z)=(\Psi(z)^{-1}I-\Delta)^{-1},
\qquad
\widehat{\widetilde S}(z)
 =\frac{1}{z\Psi(z)}(\Psi(z)^{-1}I-\Delta)^{-1}.
\]
Taking the spatial Fourier transform yields \eqref{eq:ZY-transforms}.
Banach-valued Laplace uniqueness therefore identifies $\widetilde P(t)$ with $P_{\Aa}(t)$ for a.e. $t>0$ and initially identifies $\widetilde S(t)$ with $S_{\Aa}(t)$ almost everywhere on $L^q$, $q<\infty$.
The standard inverse-subordination construction makes $\widetilde S$ strongly continuous on these spaces, so the latter equality holds for every $t>0$.
This proves \eqref{eq:SP-convolution}.

Since $Z(t,\RR^n)$ is a probability measure and $Y(t,\RR^n)=\psi(t)$ by construction, \eqref{eq:mass-kernels} follows immediately.
Young's inequality for positive finite measures then gives \eqref{eq:positive-resolvent-bounds}.
The contour identification also gives strong continuity on $L^1$.
The $BUC$ statement follows from the same construction and the strongly continuous heat semigroup on $BUC$.

\end{proof}

\begin{remark}
\label{rem:one-dimensional-positivity}
Let $F(s)=\Psi(s)^{-1}$. When $n=1$, the conclusions of
Proposition~\ref{prop:positive-pair} remain valid if we only require that $\sqrt{F}$ is a Bernstein function.
Indeed, the one dimensional spatial resolvent formulae are
\[
 (\LT_t Y_1)(s,x)=\frac{\e^{-|x|\sqrt{F}(s)}}{2\sqrt{F}(s)},
 \qquad
 (\LT_t Z_1)(s,x)=\frac{\sqrt{F}(s)}{2s}\e^{-|x|\sqrt{F}(s)},
 \qquad s>0,
\]
Since $\sqrt{F}$ is Bernstein, $\e^{-r\sqrt{F}}$, $1/\sqrt{F}$, and $\sqrt{F}/s$ are completely monotone \cite{SchillingSongVondracek2012}, so both above expressions are completely monotone in $s$.
By Bernstein’s theorem, both inverse Laplace transforms are nonnegative measures in time.
Their spatial integrals satisfy
\[
 \int_{\RR}(\LT_t Y_1)(s,x)\,dx=\Psi(s),
 \qquad
 \int_{\RR}(\LT_t Z_1)(s,x)\,dx=\frac1s.
\]
By uniqueness of the Laplace transform, these identities give $Y_1(t,\mathbb R)=\psi(t)$ and $Z_1(t,\mathbb R)=1$ for a.e. $t>0$.
The same analyticity argument gives $\psi>0$ almost everywhere.
Identification with the contour resolvents, Young's inequality,
and Proposition~\ref{prop:linear-families} give the remaining
norm and continuity properties.

Consequently, the results for the nonnegative heat equation also hold in one dimension under this weaker condition.
Explicit sufficient certificates are provided by Corollary~\ref{cor:BF-b-certificate} with $\theta=1/2$.
\end{remark}

\subsection{Integrated memory and weak scaling}
\label{subsec:integrated-memory}

Under Assumption~\ref{ass:admissible}, define the integrated memory function $K(t)=\int_0^t\psi(s)\,ds$ for $t\ge0$.
Proposition~\ref{prop:positive-pair} ensures $\psi>0$ almost everywhere on $(0,\infty)$, so $K$ is continuous and strictly increasing on $[0,\infty)$, with Laplace transform $\LT\{K\}(z)=\frac{\Psi(z)}{z}$ for $z>0$.
Applying Karamata's Tauberian theorem to \eqref{eq:symbol-asymp} yields
\begin{equation}
\label{eq:K-asymp}
K(t)\sim\frac{t^{\alpha_1}}{\Gamma(1+\alpha_1)}
\quad(t\to 0^+),
\qquad
K(t)\sim\frac{t^{\alpha_2}}{\Gamma(1+\alpha_2)}
\quad(t\to\infty).
\end{equation}
These asymptotics show that $K$ maps $[0,\infty)$ onto itself, and we denote its inverse by $K^{-1}$.

\begin{lemma}
\label{lem:K-scaling}
For every $\beta>\max\{\alpha_1,\alpha_2\}$ there exists $c_\beta\in(0,1]$ such that
\begin{equation}
\label{eq:K-scaling}
K(s)\ge c_\beta\left(\frac{s}{t}\right)^\beta K(t),
\qquad 0<s\le t<\infty.
\end{equation}
\end{lemma}

\begin{proof}
By \eqref{eq:K-asymp}, $K$ is regularly varying at zero with index $\alpha_1$ and at infinity with index $\alpha_2$.
Choose $\varepsilon>0$ so small that $\alpha_j+\varepsilon<\beta$ for $j=1,2$.
Potter bounds give \eqref{eq:K-scaling} when $0<s\le t\le t_0$ for sufficiently small $t_0$, and when $t_1\le s\le t$ for sufficiently large $t_1$.
In the cross region $0<s\le t_0<t_1\le t$, the two asymptotic comparisons give
\[
 \frac{K(s)}{K(t)}\left(\frac{t}{s}\right)^\beta
 \ge C s^{\alpha_1-\beta}t^{\beta-\alpha_2}
 \ge C t_0^{\alpha_1-\beta}t_1^{\beta-\alpha_2}>0.
\]
As $s\to0^+$ with $t\in[t_0,t_1]$, the same normalised ratio tends to infinity uniformly in $t$. Symmetrically, as $t\to \infty$ with $s\in[t_0,t_1]$, it tends to infinity uniformly in $s$. Excluding these two tail regions leaves a compact domain $0<s_0\le s\le t\le t_2$, on which the continuity and positivity of $K$ give a positive minimum. Combining these cases proves the global bound. The regular variation results used here can be found in \cite{BinghamGoldieTeugels1987}.
\end{proof}

\section{Wellposedness and comparison}
\label{sec:wellposedness}

\subsection{Abstract semilinear Scarpi--Volterra equation}
\label{subsec:abstract-wp}

The nonlinear analysis uses only the properties of the linear families established in Proposition~\ref{prop:linear-families}.
Accordingly, let $E$ be a Banach space, let $S_{\Aa}:[0,\infty)\to\mathcal L(E)$ be strongly continuous and locally bounded with $S_{\Aa}(0)=I$, and let $P_{\Aa}\in L^1((0,T);\mathcal L(E))$ for every $T>0$.
For every \(R>0\), let \(L_R\) be such that \(\norm{\Nn(u)-\Nn(v)}_E\le L_R\norm{u-v}_E\) when \(\norm{u}_E,\norm{v}_E\le R\).
A mild solution of
\begin{equation*}
{}^{S}D_0^{\alpha(t)}u=\Aa u+\Nn(u),
\qquad u(0)=u_0,
\end{equation*}
is a function satisfying
\begin{equation}
\label{eq:abstract-mild}
u(t)=S_{\Aa}(t)u_0+\int_0^tP_{\Aa}(t-s)\Nn(u(s))d s.
\end{equation}

\begin{theorem}
\label{thm:abstract-local}
For every $u_0\in E$, there is $T_*>0$ for which \eqref{eq:abstract-mild} has a unique solution $u\in C([0,T_*];E)$.
Two mild solutions with the same initial value coincide on every common interval of existence.
\end{theorem}

\begin{proof}
Fix $T_1>0$ and let $C_0=\norm{S_{\Aa}(\cdot)u_0}_{C([0,T_1];E)}$.
Choose $\rho>0$ for $R=C_0+\rho$, and let $L_R$ be a Lipschitz constant of $\Nn$ on the ball of radius $R$.
Then $\norm{\Nn(v)}\le B_R:=\norm{\Nn(0)}+L_RR$ on that ball.
For $0<T_*\le T_1$, define
\[
 \mathcal E_{\rho,T_*}
 =\{v\in C([0,T_*];E):
 \norm{v-S_{\Aa}(\cdot)u_0}_{C([0,T_*];E)}\le\rho\}
\]
and let $\mathcal T$ denote the right-hand side of \eqref{eq:abstract-mild}.
Convolution with the locally integrable operator kernel $P_{\Aa}$ maps continuous functions to continuous functions.
If $v\in\mathcal E_{\rho,T_*}$, then
\[
 \norm{\mathcal Tv-S_{\Aa}(\cdot)u_0}_{C}
 \le B_R\int_0^{T_*}\norm{P_{\Aa}(r)}d r,
\]
whereas, for $v,w\in\mathcal E_{\rho,T_*}$,
\[
 \norm{\mathcal Tv-\mathcal Tw}_{C}
 \le L_R\int_0^{T_*}\norm{P_{\Aa}(r)}d r\,
 \norm{v-w}_{C}.
\]
Since $P_{\Aa}$ is locally integrable, $T_*$ can be chosen so that the first right-hand side is at most $\rho$ and the second coefficient is less than one.
Banach's fixed-point theorem gives a mild solution in the stated ball.

For uniqueness outside that particular ball, let $u_1,u_2$ be two solutions on $[0,T]$ and choose a Lipschitz constant $L$ on a ball containing both ranges.
Absolute continuity of the integral of $\norm{P_{\Aa}}$ allows one to choose $h>0$ with $L\int_0^h\norm{P_{\Aa}(r)}d r<1$.
If the solutions agree on $[0,t_j]$, then, for $t\in[t_j,\min\{t_j+h,T\}]$, their difference is the integral only over $[t_j,t]$.
Taking the supremum on this subinterval forces the difference to vanish there.
Finitely many steps cover $[0,T]$.
\end{proof}

\begin{proposition}
\label{prop:abstract-continuous-dependence}
Let $u$ and $v$ be mild solutions on $[0,T]$ with initial values $u_0$ and $v_0$, and suppose $\norm{u}_{C([0,T];E)}+\norm{v}_{C([0,T];E)}\le R$.
Then
\begin{equation}
\label{eq:abstract-continuous-dependence}
\norm{u-v}_{C([0,T];E)}\le C_{R,T}\norm{u_0-v_0}_E.
\end{equation}
\end{proposition}

\begin{proof}
Let $M_T=\sup_{0\le t\le T}\norm{S_{\Aa}(t)}$ and let $L_R$ be a Lipschitz constant of $\Nn$ on the relevant ball.
The difference satisfies
\[
 \norm{u(t)-v(t)}
 \le M_T\norm{u_0-v_0}
 +L_R\int_0^t\norm{P_{\Aa}(t-s)}\norm{u(s)-v(s)}d s.
\]
Choose $h>0$ so that $L_R\int_0^h\norm{P_{\Aa}(r)}d r\le1/2$.
On the first interval the recent history is absorbed, giving a bound by $2M_T\norm{u_0-v_0}$.
On every subsequent interval of length at most $h$, the already controlled old history is finite and the recent part is again absorbed.
A finite induction over $\lceil T/h\rceil$ intervals gives \eqref{eq:abstract-continuous-dependence}.
\end{proof}

\begin{lemma}
\label{lem:abstract-continuation}
Let $u\in C([0,T_0];E)$ satisfy \eqref{eq:abstract-mild}.
Then there is $\delta>0$ such that $u$ extends uniquely to $[0,T_0+\delta]$.
\end{lemma}

\begin{proof}
Fix $\delta_0>0$ and, for $0\le s\le\delta_0$, define the known history
\begin{equation}
\label{eq:abstract-history}
g_{T_0}(s)=S_{\Aa}(T_0+s)u_0
 +\int_0^{T_0}P_{\Aa}(T_0+s-r)\Nn(u(r))d r.
\end{equation}
It satisfies $g_{T_0}(0)=u(T_0)$.
To prove continuity, extend $P_{\Aa}|_{(0,T_0+2\delta_0)}$ by zero to an element of $L^1(\RR;\mathcal L(E))$.
Translation continuity in this space, together with continuity and boundedness of $\Nn(u)$, proves that the integral in \eqref{eq:abstract-history} is continuous in $s$.
Hence $g_{T_0}\in C([0,\delta_0];E)$.

Let $G_0=\norm{g_{T_0}}_{C([0,\delta_0];E)}$.
For a fixed $\rho>0$, let $R=G_0+\rho$, and let $L_R$ and $B_R$ denote the Lipschitz constant and the bound of $\Nn$ on the ball of radius $R$.
Since $P_{\Aa}$ is locally integrable, choose $0<\delta\le\delta_0$ so that
\[
 B_R\int_0^\delta\norm{P_{\Aa}(r)}d r\le\rho,
 \qquad
 L_R\int_0^\delta\norm{P_{\Aa}(r)}d r<1.
\]
On the closed set
\[
 \mathcal F_{\rho,\delta}
 =\{w\in C([0,\delta];E):
 \norm{w-g_{T_0}}_{C([0,\delta];E)}\le\rho\},
\]
define
\[
 (\mathcal Kw)(s)=g_{T_0}(s)
 +\int_0^sP_{\Aa}(s-r)\Nn(w(r))d r.
\]
These choices ensure that \(\mathcal K\) maps \(\mathcal F_{\rho,\delta}\) into itself and is a contraction there, so it has a unique fixed point $w$.
Since $w(0)=g_{T_0}(0)=u(T_0)$, define $\widetilde u(T_0+s)=w(s)$.
Substitution of \eqref{eq:abstract-history} combines the old and recent histories into the single integral in \eqref{eq:abstract-mild}.
This proves the extension and its uniqueness.
\end{proof}

\begin{theorem}
\label{thm:abstract-maximal}
For every $u_0\in E$, there is a unique maximal mild solution
\[
 u\in C([0,T_{\max});E),
 \qquad 0<T_{\max}\le\infty.
\]
If $T_{\max}<\infty$, then
\begin{equation}
\label{eq:abstract-alternative}
\limsup_{t\to T_{\max}}\norm{u(t)}_E=\infty.
\end{equation}
\end{theorem}

\begin{proof}
Let $\mathcal I$ be the set of times $T>0$ for which a mild solution exists on $[0,T]$ and define $T_{\max}=\sup\mathcal I$.
Theorem~\ref{thm:abstract-local} shows that $\mathcal I$ is nonempty.
Local uniqueness makes all such solutions compatible on overlaps, and their union defines a unique solution on $[0,T_{\max})$.
Lemma~\ref{lem:abstract-continuation} shows that no solution can have $T_{\max}$ as a finite right endpoint while already being defined continuously there.

Suppose now that $T_{\max}<\infty$ but $\sup_{t<T_{\max}}\norm{u(t)}_E<\infty$.
Then $\Nn(u)$ is bounded.
Extend it by zero outside $(0,T_{\max})$, and extend $P_{\Aa}|_{(0,T_{\max}+1)}$ by zero to the real line.
{Convolution of an $L^1$ operator kernel with a bounded $E$-valued function is uniformly continuous under translations.
Hence
\[
 t\longmapsto\int_0^tP_{\Aa}(t-s)\Nn(u(s))d s
\]
has a limit in $E$ as $t\to T_{\max}$.}
Strong continuity of $S_{\Aa}$ gives a corresponding limit of $u(t)$.
Defining $u(T_{\max})$ by that limit and applying the full-history construction of Lemma~\ref{lem:abstract-continuation} at $T_{\max}$ extends the solution beyond $T_{\max}$, a contradiction.
{This proves \eqref{eq:abstract-alternative}.}
\end{proof}

\subsection{The nonnegative Scarpi heat equation on
\texorpdfstring{$\RR^n$}{R n}}
\label{subsec:concrete-wp}

For the whole-space problem, let $X=\Xspace$ with norm $\norm{v}_X=\norm{v}_1+\norm{v}_\infty$ and, for $T>0$, define 
\begin{equation*}
\mathcal X_T=C([0,T];L^1(\RR^n))
\cap L^\infty((0,T);L^\infty(\RR^n)),
\end{equation*}
with norm
\[
 \norm{v}_{\mathcal X_T}
 =\sup_{0\le t\le T}\norm{v(t)}_1
 +\operatorname*{ess\,sup}_{0<t<T}\norm{v(t)}_\infty.
\]
This mixed space is necessary because $S_{\Aa}(t)$ is not strongly continuous at zero on all of $L^\infty$.
{If $u_0\in L^1\cap BUC$, the abstract theory applies directly in $L^1\cap BUC$.
For general $u_0\in X$, the same contraction is carried out in $\mathcal X_T$.}

We now apply the abstract theory to the concrete Scarpi--Volterra problem \eqref{eq:model} on $\RR^n$, with $\Aa=\Delta$.
The linear families $S_{\Aa}$ and $P_{\Aa}$ are given by Proposition~\ref{prop:positive-pair}.
For noninteger $p$, extend the reaction to real arguments by $\Nn(v)=(v_+)^p$.
If $\norm{v}_\infty,\norm{w}_\infty\le R$, then
\begin{equation}
\label{eq:reaction-lipschitz}
\norm{\Nn(v)-\Nn(w)}_1
+\norm{\Nn(v)-\Nn(w)}_\infty
\le pR^{p-1}\norm{v-w}_X.
\end{equation}
A mild solution of \eqref{eq:model} is a function in $\mathcal X_T$ satisfying
\begin{equation}
\label{eq:mild}
u(t)=S_{\Aa}(t)u_0+\int_0^tP_{\Aa}(t-s)(u(s)_+)^pd s
\end{equation}
in $L^1$ for every $t\in[0,T]$ and in $L^\infty$ for a.e. $t>0$.

\begin{theorem}
\label{thm:wellposedness}
Assume Assumption~\ref{ass:admissible}, let $u_0\in X$, $u_0\ge0$, and let $p>1$.
There is a unique nonnegative maximal mild solution of \eqref{eq:mild} such that
\begin{equation*}
u\in C([0,T_{\max});L^1),
\qquad
\operatorname*{ess\,sup}_{0<t<T}\norm{u(t)}_\infty<\infty
\quad\text{for every }T<T_{\max},
\end{equation*}
where $0<T_{\max}\le\infty$.
If $u_0\in L^1\cap BUC$, then $u\in C([0,T_{\max});L^1\cap BUC)$.
If $0\le u_0\le v_0$, the corresponding solutions satisfy $u(t)\le v(t)$ almost everywhere on their common existence interval.
On any common finite interval on which two solutions remain bounded in $\mathcal X_T$, they depend continuously on their initial condition.
\end{theorem}

\begin{proof}
By Proposition~\ref{prop:positive-pair},
\[
 \norm{S_{\Aa}(t)g}_X\le\norm{g}_X,
\qquad
 \norm{P_{\Aa}(t)g}_X\le\psi(t)\norm{g}_X
\]
for the relevant representatives.
The linear term belongs to $\mathcal X_T$ and attains $u_0$ strongly in $L^1$.
Fix $\rho>0$, let $R=\norm{u_0}_X+\rho$, and consider the closed ball in $\mathcal X_T$ centred at $S_{\Aa}(\cdot)u_0$ with radius $\rho$.
Every function in this ball has $X$-norm bounded by $R$ almost everywhere.
The mild map $\mathcal T$ therefore satisfies
\[
 \norm{\mathcal Tv-S_{\Aa}(\cdot)u_0}_{\mathcal X_T}
 \le C R^p K(T)
\]
and, by \eqref{eq:reaction-lipschitz},
\[
 \norm{\mathcal Tv-\mathcal Tw}_{\mathcal X_T}
 \le pR^{p-1}K(T)\norm{v-w}_{\mathcal X_T}.
\]
The $L^1$ component of the convolution is continuous because $t\mapsto(v(t)_+)^p$ is continuous in $L^1$ on a uniformly bounded $L^\infty$ range and $P_{\Aa}$ is locally integrable as an $L^1$ operator kernel.
Its $L^\infty$ component is essentially bounded by the second estimate in \eqref{eq:positive-resolvent-bounds}.
Since $K(T)\downarrow0$ as $T\to0^+$, the map is contractive for small $T$.
This proves local existence and uniqueness.
The finite-step uniqueness and continuous dependence arguments of Theorem~\ref{thm:abstract-local} and Proposition~\ref{prop:abstract-continuous-dependence} apply with $\int_0^h\norm{P_{\Aa}(r)}d r$ replaced by $K(h)$.

Starting Picard iteration from $S_{\Aa}(t)u_0\ge0$ gives an increasing nonnegative sequence because $S_{\Aa}(t)$ and $P_{\Aa}(t)$ are positive and $r\mapsto r^p$ is increasing on $[0,\infty)$.
Its contraction limit is nonnegative.
The same paired iteration proves comparison locally.
At each continuation step, the positivity of $P_{\Aa}$ and the monotonicity of the reaction term preserve the ordering of the known histories, allowing the argument to extend throughout the common interval of existence.
For continuation at $T_0$, define the history as in \eqref{eq:abstract-history} with $\Nn(u)=u^p$.
Its $L^1$ component is continuous: $u^p$ is continuous in $L^1$ on bounded $L^\infty$ ranges by \eqref{eq:reaction-lipschitz}, and the zero extension of $P_{\Aa}$ is translation-continuous in $L^1((0,T);\mathcal L(L^1))$.
Its $L^\infty$ component is bounded by \eqref{eq:positive-resolvent-bounds}.
The same mixed-space contraction therefore extends every solution whose $X$-norm is uniformly bounded up to the continuation time.
Compatible local solutions combine into the unique maximal solution.
The additional $BUC$ statement follows from the strong continuity and invariance of the heat resolvent families on $BUC$.
\end{proof}

\begin{proposition}
\label{prop:mass-alternative}
For the nonnegative maximal solution $u(x,t)$ in Theorem~\ref{thm:wellposedness}, the mass is \(\int_{\mathbb R^n}u(x,t)\,dx=\|u(t)\|_1\), which satisfies
\begin{equation}
\label{eq:mass-identity}
\norm{u(t)}_1=\norm{u_0}_1
 +\int_0^t\psi(t-s)\norm{u(s)}_p^pd s,
\qquad 0<t<T_{\max}.
\end{equation}
Moreover, if $T_{\max}<\infty$, then
\begin{equation}
\label{eq:blowup-alternative}
\operatorname*{ess\,sup}_{T<t<T_{\max}}\norm{u(t)}_\infty=\infty
\qquad\text{for every }0<T<T_{\max}.
\end{equation}
\end{proposition}

\begin{proof}
Integrating the mild formula over space and using Tonelli’s theorem and the mass identities \eqref{eq:mass-kernels} give \eqref{eq:mass-identity}.

Suppose instead that $T_{\max}<\infty$ and that $\norm{u(t)}_\infty$ is essentially bounded by $M>0$ on some terminal interval $(t_0,T_{\max})$ and hence, together with local boundedness, on all of $(0,T_{\max})$.
Then
\[
 \norm{u(t)}_p^p\le M^{p-1}\norm{u(t)}_1
\]
for a.e. $t\in(0,T_{\max})$.
$\norm{u(t)}_1$ satisfies a linear Volterra--Gronwall inequality with the nonnegative locally integrable kernel $\psi$.
Choose $h>0$ with $M^{p-1}K(h)<1/2$.
Taking suprema first on $[0,h]$ and then successively on finitely many subintervals of length $h$ covering $(0,T_{\max})$ yields a uniform bound on $\norm{u(t)}_1$ over $(0,T_{\max})$: at each step, the integral over the preceding subintervals is already finite, whereas the contribution from the current subinterval is absorbed by the factor $1/2$.

Thus a uniform essential $L^\infty$ bound gives a uniform $X$ bound.
The continuation argument in Theorem~\ref{thm:wellposedness} would extend the solution beyond $T_{\max}$, proving \eqref{eq:blowup-alternative}.
\end{proof}

\subsection{Weak Volterra formulation}
\label{subsec:weak-identity}

\begin{proposition}
\label{prop:weak-identity}
Let $u$ be the solution in Theorem~\ref{thm:wellposedness}.
For every $v\in\mathcal S(\RR^n)$ and $0<t<T_{\max}$,
\begin{equation}
\label{eq:weak-identity}
\ip{u(t)}{v}
=\ip{u_0}{v}
+\int_0^t\psi(t-s)
\left(\ip{u(s)}{\Delta v}+\ip{u(s)^p}{v}\right)d s.
\end{equation}
\end{proposition}

\begin{proof}
Fix $T<T_{\max}$ and extend $f(t)=u(t)^p$ by zero for $t\ge T$.
Define for all $t\ge0$
\[
 \bar u(t)=S_{\Aa}(t)u_0+\int_0^tP_{\Aa}(t-s)f(s)d s.
\]
Then $\bar u=u$ on $[0,T]$.
The fixed-contour bounds for $S_{\Aa}$ and $P_{\Aa}$, together with the compact temporal support of $f$, make $\bar u$ of exponential order.
The operator Laplace formulae in \eqref{eq:PS-Laplace}, with $\Aa=\Delta$, therefore give in $\mathcal S'(\RR^n)$
\[
 \widehat{\bar u}(z)
 =\frac{1}{z\Psi(z)}(\Psi(z)^{-1}I-\Delta)^{-1}u_0
 +(\Psi(z)^{-1}I-\Delta)^{-1}\widehat{f}(z).
\]
Multiplying by $\Psi(z)^{-1}I-\Delta$ and then by $\Psi(z)$ yields
\[
 \widehat{\bar u}(z)-\frac{u_0}{z}
 =\Psi(z)\bigl(\Delta\widehat{\bar u}(z)+\widehat{f}(z)\bigr).
\]
Laplace inversion gives $\bar u-u_0=\psi*(\Delta\bar u+f)$ in space--time distributions.
Restricting to $0<t<T$, pairing with $v$, and moving $\Delta$ onto $v$ proves \eqref{eq:weak-identity}.
The pairings and Fubini interchanges are justified because, on every compact subinterval, $u,u^p\in L^1\cap L^\infty$ and $\psi\in L^1_{\mathrm{loc}}$.
\end{proof}
\begin{remark}
  For the constant case $\alpha(t)\equiv\alpha$, $\Psi(z)=z^{-\alpha}$, and the Laplace inversion gives $\bar u-u_0=\psi_\alpha*(\Delta\bar u+f)$ with $\psi_\alpha(t)=\frac{t^{\alpha-1}}{\Gamma(\alpha)}$, which is the standard weak Volterra formulation of the constant-order Caputo equation.
\end{remark}

\section{Gaussian moments and finite-time blow-up}
\label{sec:blowup}
Kaplan's weighted-moment method converts the growth of a nonnegative PDE solution into a scalar nonlinear inequality \cite{Kaplan1963}.
Its extension to semilinear equations with nonlocal temporal kernels is developed on bounded domains in \cite{VergaraZacher2017}.
We use a normalised Gaussian, which is the heat kernel at a freely chosen spatial scale.
Its unit mass closes the reaction term by Jensen's inequality, while an explicit lower bound for its Laplacian quantifies the diffusive loss.

\begin{lemma}
\label{lem:scalar-obstruction}
Let $k(t)\ge0$ be locally integrable, let $K(t)=\int_0^t k(s)d s$, and suppose that, for some $\beta>0$ and $c_\beta>0$,
\[
K(s)\ge c_\beta(s/t)^\beta K(t),
\qquad 0<s\le t.
\]
If a finite continuous function $y:[0,T]\to(0,\infty)$ satisfies
\begin{equation}
\label{eq:scalar-ineq}
y(t)\ge y_0+a\int_0^t k(t-s)y(s)^pd s,
\qquad a>0,\quad p>1,
\end{equation}
with $y_0>0$, then
\begin{equation}
\label{eq:scalar-necessary}
a c_\beta K(T)y_0^{p-1}
\le 2^{\beta p/(p-1)}.
\end{equation}
Consequently, if $K(t)\to\infty$, no globally finite function can satisfy \eqref{eq:scalar-ineq}.
\end{lemma}

\begin{proof}
The conclusion is immediate for $K(T)=0$, so we assume $K(T)>0$.
Let $t_j=T(1-2^{-j})$, $j\ge0$, and $m_j=\inf_{t_j\le t\le T}y(t)$.
For $t\in[t_{j+1},T]$, positivity and \eqref{eq:scalar-ineq} give
\[
y(t)\ge a\int_{t_j}^{t}k(t-s)y(s)^pd s
\ge aK(t-t_j)m_j^p
\ge aK(T2^{-j-1})m_j^p.
\]
Hence, with $\Lambda_T=ac_\beta K(T)$,
\[m_{j+1}\ge \Lambda_T2^{-\beta(j+1)}m_j^p,
\qquad m_0\ge y_0.\]
Writing $x_j=\log m_j$, we have
\begin{equation}
\label{eq:m-iteration}
x_{j+1}\ge\log \Lambda_T-\beta(j+1)\log2+px_j.
\end{equation}
Iterating \eqref{eq:m-iteration} gives 
\[\frac{x_j}{p^j} \ge x_0+\sum_{i=1}^{j}p^{-i}\left(\log \Lambda_T-\beta i \log2\right).\]
And thus,
\[
\liminf_{j\to\infty}\frac{x_j}{p^j}
\ge \log y_0+\frac{\log \Lambda_T}{p-1}
-\frac{\beta p\log2}{(p-1)^2}.
\]
If the right-hand side were positive, then $m_j\to\infty$, contradicting $m_j\le y(T)<\infty$.
Rearrangement gives \eqref{eq:scalar-necessary}.
\end{proof}

We next choose a whole-space analogue of the positive eigenfunction used in Kaplan's bounded-domain argument.
For $R>0$, we use the normalised Gaussian as a spatial test function:
\begin{equation*}
\xi_R(x)=(4\pi R^2)^{-n/2}
\exp\!\left(-\frac{|x|^2}{4R^2}\right),
\qquad \int_{\RR^n}\xi_Rd x=1.
\end{equation*}
It satisfies
\begin{equation}
\label{eq:gaussian-laplacian}
\Delta\xi_R
=\left(\frac{|x|^2}{4R^4}-\frac{n}{2R^2}\right)\xi_R
\ge-\frac{n}{2R^2}\xi_R.
\end{equation}
Let
\begin{equation*}
F_R(t)=\int_{\RR^n}u(x,t)\xi_R(x)d x,
\qquad
F_{R,0}=F_R(0).
\end{equation*}

\begin{proposition}
\label{prop:gaussian-ineq}
For $0<t<T_{\max}$,
\begin{equation}
\label{eq:FR-ineq}
F_R(t)\ge F_{R,0}
+\int_0^t\psi(t-s)
\left(F_R(s)^p-\frac{n}{2R^2}F_R(s)\right)d s.
\end{equation}
If
\begin{equation}
\label{eq:gaussian-threshold}
F_{R,0}^{p-1}\ge\frac{n}{R^2},
\end{equation}
then $F_R(t)\ge F_{R,0}$ and
\begin{equation}
\label{eq:FR-pure}
F_R(t)\ge F_{R,0}
+\frac12\int_0^t\psi(t-s)F_R(s)^pd s.
\end{equation}
\end{proposition}

\begin{proof}
Use $v=\xi_R$ in Proposition~\ref{prop:weak-identity}.
Equation \eqref{eq:gaussian-laplacian} and Jensen's inequality for the probability density $\xi_R$ give
\[
\ip{u}{\Delta\xi_R}\ge-\frac{n}{2R^2}F_R,
\qquad
\int u^p\xi_Rd x\ge F_R^p,
\]
which proves \eqref{eq:FR-ineq}.
Since \(u\in C([0,T];L^1)\) and \(\xi_R\in L^\infty\), the function \(F_R\) is continuous.
Fix $2^{-1/(p-1)}<\vartheta<1$. By continuity, there exists $\delta>0$ such that $F_R(t)\geq\vartheta F_{R,0}$ on $[0,\delta]$, then on this interval,
\[
F_R(t)^{p-1}
\ge
\vartheta^{p-1}F_{R,0}^{p-1}
>
\frac{n}{2R^2},
\]
so the integrand in \eqref{eq:FR-ineq} is nonnegative. Since $\psi>0$ almost everywhere, $F_R(t)>F_{R,0}$ on $(0,\delta]$.

Suppose that a first later contact \(t_*>0\) with \(F_R(t_*)=F_{R,0}\) exists.
Then \(F_R(s)>F_{R,0}\) for \(0<s<t_*\), and
\begin{equation}
  \label{eq:FR-ineq-1}
  F_R(s)^p-\frac{n}{2R^2}F_R(s)
\ge
\frac12F_R(s)^p>0.
\end{equation}
Since \(\psi(t_*-s)>0\) for a.e. \(s\in(0,t_*)\), equation~\eqref{eq:FR-ineq} gives \(F_R(t_*)>F_{R,0}\), a contradiction.
Therefore, \(F_R(t)\ge F_{R,0}\)  for all \(t\in[0,T_{\max})\). Thus, \eqref{eq:FR-ineq-1} is valid everywhere, substituting it into \eqref{eq:FR-ineq} establishes \eqref{eq:FR-pure}.
\end{proof}
We next apply the Volterra inequality satisfied by $F_R$ to derive an upper bound on the maximal lifespan in terms of the initial Gaussian moment $F_{R,0}$.
This estimate allows us either to take $R>0$ sufficiently large for any fixed nontrivial $u_0\ge0$ in the subcritical regime $1<p<1+2/n$, or to fix $R$ and scale up the amplitude of a given initial profile for arbitrary $p>1$.

\begin{theorem}
\label{thm:gaussian-blowup}
Under the Assumption~\ref{ass:admissible}, fix $\beta>\max\{\alpha_1,\alpha_2\}$ and let $c_\beta$ be given by Lemma~\ref{lem:K-scaling}.
If \eqref{eq:gaussian-threshold} holds for some $R>0$, then $T_{\max}<\infty$ and
\begin{equation}
\label{eq:gaussian-lifespan-upper}
K(T_{\max})F_{R,0}^{p-1}
\le \frac{2^{1+\beta p/(p-1)}}{c_\beta}.
\end{equation}
Equivalently,
\begin{equation*}
T_{\max}\le
K^{-1}\!\left(C_{p,\beta}F_{R,0}^{1-p}\right),\qquad C_{p,\beta} = 2^{1+\beta p/(p-1)}/c_\beta.
\end{equation*}
\end{theorem}

\begin{proof}
For every $T<T_{\max}$, apply Lemma~\ref{lem:scalar-obstruction} to \eqref{eq:FR-pure} with $k=\psi$ and $a=\frac12$.
This gives \eqref{eq:gaussian-lifespan-upper} with $T$ in place of $T_{\max}$.
If the solution were global, $K(T)\to\infty$ would contradict that bound.
Hence $T_{\max}<\infty$, and letting $T\to T_{\max}$ proves the stated estimate.
\end{proof}

For a constant-order Caputo derivative,
\[
K_\alpha^{-1}(y) = \bigl(\Gamma(1+\alpha)y\bigr)^{1/\alpha},
\]
so Theorem~\ref{thm:gaussian-blowup} recovers the classical power-law lifespan upper bound.
Because the spatial scale $R>0$ in \eqref{eq:gaussian-threshold} is arbitrary, we have the following consequences for blow-up.

\begin{corollary}
\label{cor:blowup-consequences}
Let $u_0\in X$ be nonnegative and nontrivial.
\begin{enumerate}[label=(\roman*)]
\item If $1<p<1+2/n$, then $T_{\max}<\infty$.
\item If $p=1+2/n$ and $\norm{u_0}_1>(4\pi n)^{n/2}$, then $T_{\max}<\infty$.
\item For every $p>1$ and every nonzero $0\le\varphi\in X$, the datum $u_0=A\varphi$ blows up for all sufficiently large $A$.
\end{enumerate}
\end{corollary}

\begin{proof}
By using the dominated convergence,we have
\begin{equation}
\label{eq:FR-asymp-R}
R^nF_{R,0}\longrightarrow
(4\pi)^{-n/2}\norm{u_0}_1
\quad \text{ as } R\to\infty.
\end{equation}
Thus $R^2F_{R,0}^{p-1}$ grows like $R^{2-n(p-1)}$.
It tends to infinity in the subcritical range, proving (i).
At $p=1+2/n$, its limit is $\norm{u_0}_1^{2/n}/(4\pi)$, which proves (ii).
{For (iii), fix any $R$ for which $\int\varphi\xi_Rd x>0$.
Multiplication by a sufficiently large $A$ then ensures \eqref{eq:gaussian-threshold}.}
\end{proof}

\begin{remark}
The subcritical range agrees with the classical Fujita problem and the whole-space Caputo results \cite{Fujita1966,ZhangSun2015,cortazar2024semilinear}.
It follows from the balance between the $R^{-2}$ diffusion loss and the $R^{-n}$ decay of the initial Gaussian moment.
\end{remark}

\section{Lifespan bounds and memory scales}
\label{sec:lifespan}
Section~\ref{sec:blowup} determines whether blow-up occurs and provides an upper bound on the maximal lifespan.
We now quantify how this lifespan depends on the amplitude of the initial datum.
The lower bound follows from the local contraction argument, while the upper bound is derived from the Gaussian blow-up criterion.
Both bounds are expressed in terms of $K^{-1}$. 
For the Scarpi transition, regular-variation inversion gives
\[
K^{-1}(y)
\sim
\bigl(\Gamma(1+\alpha_1)y\bigr)^{1/\alpha_1}
\quad(y\to0^+),
\]
and
\[
K^{-1}(y)
\sim
\bigl(\Gamma(1+\alpha_2)y\bigr)^{1/\alpha_2}
\quad(y\to\infty).
\]

Let $u_0=A\varphi$ with $0\le\varphi\in X, \varphi\not\equiv0$ and denote the maximal lifespan by \(T_A\).

\begin{proposition}
\label{prop:lower-lifespan}
There is $C>0$, depending only on $p$ and $\norm{\varphi}_X$, such that
\begin{equation}
\label{eq:lower-intrinsic}
T_A\ge K^{-1}(CA^{1-p})
\quad \text{ for } A>0.
\end{equation}
Consequently,
\begin{equation}
\label{eq:lower-powers}
T_A\gtrsim A^{-(p-1)/\alpha_1}
\quad(A\to\infty),
\qquad
T_A\gtrsim A^{-(p-1)/\alpha_2}
\quad(A\to0^+).
\end{equation}
\end{proposition}

\begin{proof}
Apply the contraction argument of Theorem~\ref{thm:wellposedness} on the ball
\[
\norm{u}_{\mathcal X_T}
\le 2A\norm{\varphi}_X.
\]
The map is invariant and contractive whenever $C_p(A\norm{\varphi}_X)^{p-1}K(T)<1$.
{This proves \eqref{eq:lower-intrinsic}.
The bounds in \eqref{eq:lower-powers} follow from \eqref{eq:K-asymp} and inversion of regularly varying functions.}
\end{proof}

\begin{theorem}
\label{thm:large-amplitude}
For every nonzero $0\le\varphi\in X$ and every $p>1$, there are constants $A_0,C_1,C_2>0$ such that
\begin{equation}
\label{eq:matching-intrinsic}
K^{-1}(C_1A^{1-p})
\le T_A\le
K^{-1}(C_2A^{1-p}),
\qquad A\ge A_0.
\end{equation}
In particular,
\begin{equation}
\label{eq:matching-power}
T_A\asymp A^{-(p-1)/\alpha_1},
\quad \text{ as } A\to\infty.
\end{equation}
\end{theorem}

\begin{proof}
The lower bound is Proposition~\ref{prop:lower-lifespan}.
Choose $R>0$ with $m_R=\int\varphi\xi_Rd x>0$.
For $A$ sufficiently large, $(Am_R)^{p-1}\ge n/R^2$, so Theorem~\ref{thm:gaussian-blowup} gives
\[
T_A\le K^{-1}\!\left(C_{p,\beta}m_R^{1-p}A^{1-p}\right).
\]
This proves \eqref{eq:matching-intrinsic}.
Since its inverse-function arguments tend to zero as $A\to\infty$, the short-time asymptotic in \eqref{eq:K-asymp} gives \eqref{eq:matching-power}.
\end{proof}

The power-law scaling $A^{-(p-1)/\alpha_1}$ matches the reaction lifespan scale of the constant-order Caputo equation of order $\alpha_1$ \cite{FengLiLiuXu2018}, reflecting that large initial data trigger blow-up before the memory transition to $\alpha_2$ takes effect.
In the subcritical small-amplitude regime, by contrast, the initial Gaussian moment $F_{R,0}$ for any fixed radius $R$ shrinks with the amplitude and eventually violates \eqref{eq:gaussian-threshold}. To apply Theorem~\ref{thm:gaussian-blowup}, we therefore let $R$ grow as the amplitude decreases.

\begin{theorem}
\label{thm:small-upper}
Let $1<p<1+2/n$ and $d=2-n(p-1)>0$.
For $u_0=A_\varepsilon\varphi$ with $0\le\varphi\in X$, $\varphi\not\equiv0$, there are $A_{\varepsilon_0},C>0$ such that
\begin{equation}
\label{eq:small-upper-intrinsic}
T_\varepsilon\le
K^{-1}\!\left(C A_\varepsilon^{-2(p-1)/d}\right),
\qquad 0<A_\varepsilon\le A_{\varepsilon_0}.
\end{equation}
Hence
\begin{equation}
\label{eq:small-upper-power}
T_\varepsilon\lesssim
A_\varepsilon^{-\frac{1}{\alpha_2(1/(p-1)-n/2)}}
\qquad(A_\varepsilon\to0^+).
\end{equation}
\end{theorem}

\begin{proof}
By monotone convergence, we can choose $R_0\ge1$ such that
\[
 \int_{|x|\le R_0}\varphi(x)\,dx
 \ge\frac12\norm{\varphi}_1.
\]
For $R\ge R_0$, the exponential factor in $\xi_R$ is bounded below by $e^{-1/4}$ on $|x|\le R_0$, so we define
\[
m_R:=\int\varphi\xi_Rd x
\ge (4\pi)^{-n/2}R^{-n}e^{-1/4}
\int_{|x|\le R_0}\varphi(x)\,d x
\ge C_*\norm{\varphi}_1 R^{-n},
\qquad C_*=\frac{e^{-1/4}}{2(4\pi)^{n/2}}.
\]
Let \( L_*=(n/(C_*\norm{\varphi}_1)^{p-1})^{1/d} \) and $A_{\varepsilon_0}=\min\left\{1, (L_*/R_0)^{d/(p-1)}\right\}.$ For $0<A_\varepsilon\le A_{\varepsilon_0}$, by choosing $R_\varepsilon=L_*A_\varepsilon^{-(p-1)/d}$ we conclude that $R_\varepsilon\ge R_0.$ Then
\[
R_\varepsilon^2(A_\varepsilon m_{R_\varepsilon})^{p-1}
\ge C^{p-1}_*\norm{\varphi}_1^{p-1}L_*^d\ge n,
\]
so Theorem~\ref{thm:gaussian-blowup} applies.
Moreover,
\[
(A_\varepsilon m_{R_\varepsilon})^{1-p}
\le C A_\varepsilon^{1-p} R_\varepsilon^{n(p-1)}
=C A_\varepsilon^{-2(p-1)/d},
\]
which gives \eqref{eq:small-upper-intrinsic}.
Its argument tends to infinity, so the long-time asymptotic in \eqref{eq:K-asymp} yields \eqref{eq:small-upper-power}.
\end{proof}

Theorem~\ref{thm:large-amplitude} gives the matching large-amplitude lifespan bounds and shows that the initial order \(\alpha_1\) determines the exponent \(-(p-1)/\alpha_1\).
For small amplitudes in the subcritical range, Proposition~\ref{prop:lower-lifespan} and Theorem~\ref{thm:small-upper} yield
\[
 C_1 A^{-(p-1)/\alpha_2}
 \le T_A
 \le C_2 A^{-1/[\alpha_2(1/(p-1)-n/2)]},
 \qquad A\to0^+.
\]
These estimates bound $T_A$ between two different power-law scalings: the lower bound relies only on the boundedness of the resolvent families, whereas the upper bound exploits an amplitude-dependent Gaussian radius.

\section{Numerical study of growth and memory effects}
\label{sec:numerics}

In this section, we examine how the two endpoint orders $\alpha_1, \alpha_2$ and the transition rate $c$ affect nonlinear growth and spatial profiles. All numerical experiments for the nonlinear problem are performed in one spatial dimension with $1 < p < 3$, corresponding to the subcritical range established in Corollary~\ref{cor:blowup-consequences}. Within this subcritical regime, we focus on the growth dynamics and the transition between the short- and long-time memory scales. Unless stated otherwise, we fix $(\alpha_1,\alpha_2)=(0.6,0.7)$ and $c=1$, and vary $c\in\{10^{-1/4},1,10\}$ in the transition-rate comparisons. All three parameter triples $(\alpha_1,\alpha_2,c)$ satisfy Assumption~\ref{ass:admissible}, as certified by the positive margins in Table~\ref{tab:BF-b-certificates}. 

\subsection{Fourier discretisation and convolution quadrature}

We discretise the Volterra formulation \eqref{eq:abstract-volterra-form} with $\Aa=\Delta$ and $\Nn(u)=u^p$, using a Fourier approximation in space and fully implicit backward-Euler convolution quadrature (BE--CQ) in time.
The temporal approximation follows Lubich's fractional multistep construction \cite{Lubich1986} and the convolution quadrature framework \cite{Lubich1988I,Lubich1988II}, in which the weights are generated directly from the Laplace transform of the convolution kernel.
Its application to linear and semilinear fractional evolution equations is analysed in \cite{CuestaLubichPalencia2006, WangZhou2020}.

The spatial approximation uses Fourier differentiation and projection on the periodic interval $(-l,l)$ \cite{ShenTangWang2011}, with
\[
 \Delta x=\frac{2l}{M},
 \qquad
 x_j=-l+j\Delta x,
 \qquad j=0,\ldots,M-1.
\]
The Fourier second derivative $D_{xx,M}$ has multiplier $-(\pi k/l)^2$.
The nonlinear term is approximated by
\[
 \mathcal N_M(V)=P_M[(I_{3M/2}V)_+^p],
\]
where $I_{3M/2}$ interpolates to the grid with $3M/2$ points and $P_M$ projects back to the retained Fourier modes.

For $t_m=mh$, $m=0,\ldots,N$, substituting the backward-Euler symbol $(1-\zeta)/h$ into $\Psi=\widehat\psi$ defines the convolution weights:
\[
 \Psi\!\left(\frac{1-\zeta}{h}\right)
 =\sum_{m\ge0}\omega_m\zeta^m.
\]
With $U_j^0=u_0(x_j)$, the fully discrete equation is
\begin{equation}
\label{eq:exact-CQ}
 U^m=U^0+\sum_{j=0}^m\omega_{m-j}
 \bigl[D_{xx,M}U^j+\mathcal N_M(U^j)\bigr],
 \qquad m=1,\ldots,N.
\end{equation}
The sum retains the initial contribution and the complete discrete history.
Both diffusion and reaction are implicit at the current time level.

At each time step, we solve the implicit equation by damped Newton iteration. Each Newton step uses GMRES with the Fourier preconditioner $I-\omega_0D_{xx,M}$.

The weights are evaluated by Cauchy's coefficient formula with FFT acceleration:
\[
 \omega_m\approx\frac{\rho^{-m}}{Q}\operatorname{Re}
 \sum_{q=0}^{Q-1}
 \Psi\!\left(\frac{1-\rho e^{2\pi iq/Q}}{h}\right)e^{-2\pi imq/Q},
 \qquad 0\le m\le N.
\]
We use $\rho=\exp(-2/N)$ and the smallest power of two $Q\ge\max\{128,16(N+1)\}$.
The coefficients are compared with those computed using $2Q$ nodes, and the latter are retained \cite{GarrappaGiusti2023}.
The principal logarithm is used in the Scarpi symbol.
The Caputo comparisons use the same discretisation with $\Psi(z)=z^{-\alpha_1}$ or $\Psi(z)=z^{-\alpha_2}$.

For the long-time rescaling below, the same scheme is applied to $v_b(y,\sigma)$ using $\Psi_b$ and a uniform grid in $\sigma$.
A step $h$ in $\sigma$ corresponds to a step $bh$ in the original time $t=b\sigma$.
The large-amplitude computations instead use a uniform grid in $t$, followed by comparison at common rescaled times.

\subsection{Numerical verification and growth thresholds}

For a spatial grid function $V$, we use
\[
 \norm{V}_2
 =\left(\Delta x\sum_{j=0}^{M-1}|V_j|^2\right)^{1/2},
 \qquad
 \norm{V}_\infty=\max_{0\le j<M}|V_j|.
\]
These norms are evaluated over the full original computational interval, including when only part of a profile is displayed.
Scarpi--Caputo relative $L^2$ differences are normalised by the Caputo norm at the same time.
In time-step refinement tests, we use the finer solution to compute the reference norm.
To compare solutions accurately, we restrict enlarged-domain solutions to the original interval, evaluate them at common spatial nodes, and apply linear interpolation in time.

For $\Lambda\in\{2,4,8\}$, let \(\tau_\Lambda\) be the first time at which \(|U(t)|_\infty\) crosses the level \(\Lambda|U^0|_\infty\) from below.
The crossing time is obtained by linear interpolation of the logarithm of the peak, while spatial profiles are interpolated linearly in time.
These times compare prescribed relative growth levels and are distinct from the maximal analytical lifespan.
All integrated-memory comparisons use the continuous function $K$, denoted by $K_{\rm model}$ when several equations are compared.

We initially verify the temporal discretisation using an exact linear solution. For the Caputo model ($\alpha=1/2$, $u_0=\cos x$ on $(-\pi,\pi)$), the exact solution is
\[
 u(x,t)=\e^t\operatorname{erfc}(\sqrt t)\cos x.
\]
Using $M=64$ and $h=2^{-7},2^{-8},2^{-9},2^{-10}$, the $L^2$ error at $T=1$ decreases from $1.05\times10^{-3}$ to $1.18\times10^{-4}$.
For constant order, the computed weights agree to round-off with the recurrence $\omega_0=h^\alpha$, $\omega_m=(m-1+\alpha)\omega_{m-1}/m$.

For the nonlinear test, let $p=2$ and $u_0=4\mathrm{e}^{-x^2}$ on $(-l,l)=(-8,8)$ with $M=256$. For the Scarpi equation and both Caputo equations, we solve the problems on $t\in[0,0.02]$ using time steps $h=0.02\times2^{-k}$ ($k=10,11,12$). We compare the $k=11$ and $k=12$ solutions at common times and normalise their $L^2$ difference by the norm of the $k=12$ solution. The maximum relative difference over the comparison times and the three models is $1.93\times10^{-4}$. At the finest time step, we also compute solutions on $(l,M)=(8,512)$ and $(16,1024)$. The first grid halves the original mesh spacing. The second doubles the domain at the finer spacing.

In subsequent experiments, we halve $h$ for temporal refinement, double $M$ at fixed $l$ for spatial refinement, and double both $l$ and $M$ at fixed $\Delta x$ for domain enlargement.
Our convergence criteria require relative threshold changes below $0.5\%$ and relative $L^2$ solution changes below $10^{-3}$.

In the large-amplitude comparisons, we directly evaluate the refinement change in the Scarpi--Caputo solution difference, normalised by the $L^2$ norm of the Caputo solution. We first maximise each refinement change over time, then take the larger of the temporal and spatial values and add the domain-enlargement contribution. For the long-time rescaling, Table~\ref{tab:endpoint-terminal} reports the combined refinement changes of the individual solutions.

We use Fourier--Euler inversion for the kernels and relaxation functions, with independent Gaver--Stehfest checks and additional de Hoog checks at selected points.

\subsection{Memory and linear relaxation}

Figure~\ref{fig:endpoint-memory} compares the continuous $K$ with \eqref{eq:K-asymp} through
\[
 Q_{K,1}(t)=\frac{\Gamma(1+\alpha_1)K(t)}{t^{\alpha_1}},
 \qquad
 Q_{K,2}(t)=\frac{\Gamma(1+\alpha_2)K(t)}{t^{\alpha_2}}.
\]
The first ratio approaches 1 at short times and the second at long times.
The logarithmic slope
\[
 a_K(t)=\frac{d\log K(t)}{d\log t}
       =\frac{t\psi(t)}{K(t)}
\]
falls below $\alpha_1$ and later exceeds $\alpha_2$ for each transition rate tested, although the prescribed order $\alpha(t)$ increases monotonically.

For $\lambda>0$, let $R_S$ and $R_{C_\alpha}$ denote the Scarpi and Caputo relaxation amplitudes with initial value $1$.
The function $R_S$ is the time-dependent amplitude of a Fourier component of the linear heat equation, satisfying ${}^S D_0^{\alpha(t)}R_S=-\lambda R_S$, with $R_S(0)=1$ and $\lambda=|\xi|^2$.
Then
\[
 \widehat R_S(z)=\frac1{z(1+\lambda\Psi(z))},
 \qquad
 R_{C_\alpha}(t)=E_\alpha(-\lambda t^\alpha),
\]
where $E_\alpha$ is the Mittag--Leffler function \cite{Diethelm2010}. The relaxation ratios
\[
 \frac{1-R_S(t)}{1-R_{C_{\alpha_1}}(t)},
 \qquad
 \frac{R_S(t)}{R_{C_{\alpha_2}}(t)}
\]
approach one as $t\to0$ and $t\to\infty$, respectively, by the asymptotics of $\Psi$ and the monotone Tauberian theorem for the relaxation functions. Thus $\alpha_1$ determines the initial departure from one and $\alpha_2$ the long-time decay.
$1-R_S$ is computed by direct numerical inversion of its Laplace transform.
Related kernel and relaxation comparisons for exponential Scarpi transitions were studied in \cite{GarrappaGiustiMainardi2021,GarrappaGiusti2023}.

\begin{figure}[htbp]
\centering
\includegraphics[width=0.92\textwidth]{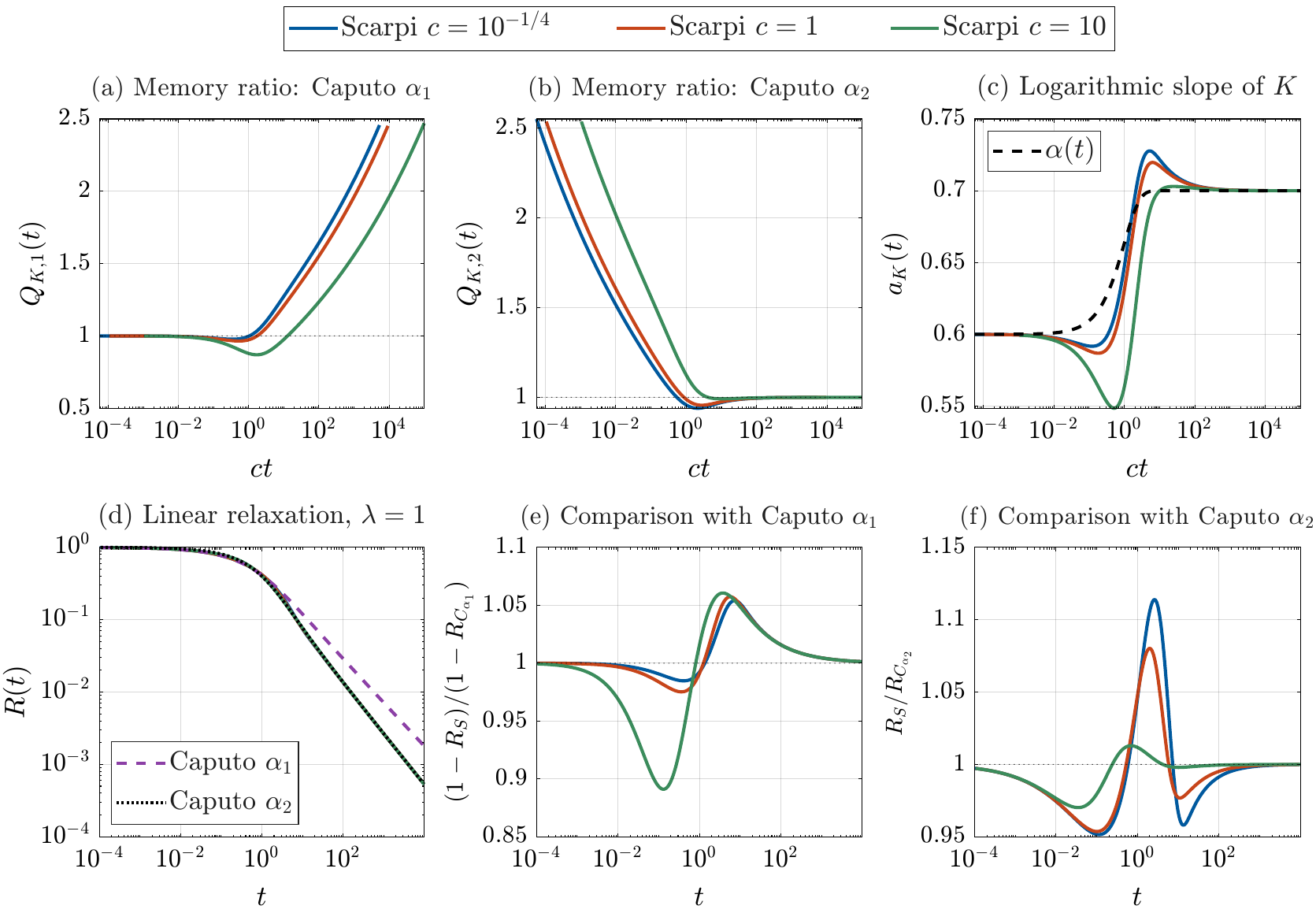}
\caption{Memory and linear relaxation at three transition rates.
Panels (a,b) show the integrated-memory ratios $Q_{K,1},Q_{K,2}$.
Panel (c) compares the logarithmic slope $a_K$ with $\alpha(t)$.
Panels (d--f) show the relaxation modes and the two endpoint ratios for $\lambda=1$.
The upper panels use $ct$, and the lower panels use $t$.}
\label{fig:endpoint-memory}
\end{figure}

\subsection{Large-amplitude growth and the initial order}

We take $u_0=A\e^{-x^2}$, $p=2$ and $A=4(3/2)^j$, $j=0,\ldots,11$.

Theorem~\ref{thm:large-amplitude} gives the lifespan scale $A^{-1/\alpha_1}$.
The same scale applies to fixed relative thresholds of the continuous solution under Assumption~\ref{ass:admissible}.
Indeed, let $\tau_\Lambda^{\rm cont}$ be the first time at which $\norm{u(t)}_\infty=\Lambda A$, with fixed $\Lambda>1$.
Before this time, the mild formula and \eqref{eq:positive-resolvent-bounds} give
\[
 \norm{u(t)}_\infty\le A+\Lambda^2A^2K(t).
\]
Continuity and Theorem~\ref{thm:large-amplitude} therefore yield, for large $A$,
\[
 \frac{\Lambda-1}{\Lambda^2}A^{-1}
 \le K(\tau_\Lambda^{\rm cont})
 \le K(T_A)\le C_2A^{-1}.
\]
Hence $\tau_\Lambda^{\rm cont}\asymp A^{-1/\alpha_1}$.
The numerical thresholds test this scale and compare the two operators at equal relative peak growth.

The Scarpi solution reaches fourfold growth slightly later than Caputo $\alpha_1$ at every sampled amplitude, and the ratio of their threshold times decreases towards one as $A$ increases.
Figure~\ref{fig:endpoint-large} shows the threshold times, their adjacent-amplitude logarithmic slopes, this ratio and $AK_{\rm model}(\tau_4)$.
Multiplication by $A$ removes the $A^{-1}$ leading scale of $K_{\rm model}(\tau_4)$.
The slopes are the changes in $\log\tau_4$ divided by the changes in $\log A$, plotted at the geometric mean of each adjacent amplitude pair.
Restricting the fit to larger amplitudes brings the slopes closer to the reference exponents in Table~\ref{tab:endpoint-fits}.

\begin{table}[htbp]
\centering
\small
\caption{Slopes of $\log\tau_4$ against $\log A$, fitted over the largest five, six or seven amplitudes.}
\begin{tabular}{ccccc}
\toprule
Model & Largest 5 & Largest 6 & Largest 7 & Reference exponent\\
\midrule
Scarpi & -1.6933 & -1.7000 & -1.7085 & $-5/3$\\
Caputo $\alpha_1$ & -1.6933 & -1.6999 & -1.7083 & $-5/3$\\
Caputo $\alpha_2$ & -1.4504 & -1.4557 & -1.4628 & $-10/7$\\
\bottomrule
\end{tabular}
\label{tab:endpoint-fits}
\end{table}

\begin{figure}[htbp]
\centering
\includegraphics[width=0.76\textwidth]{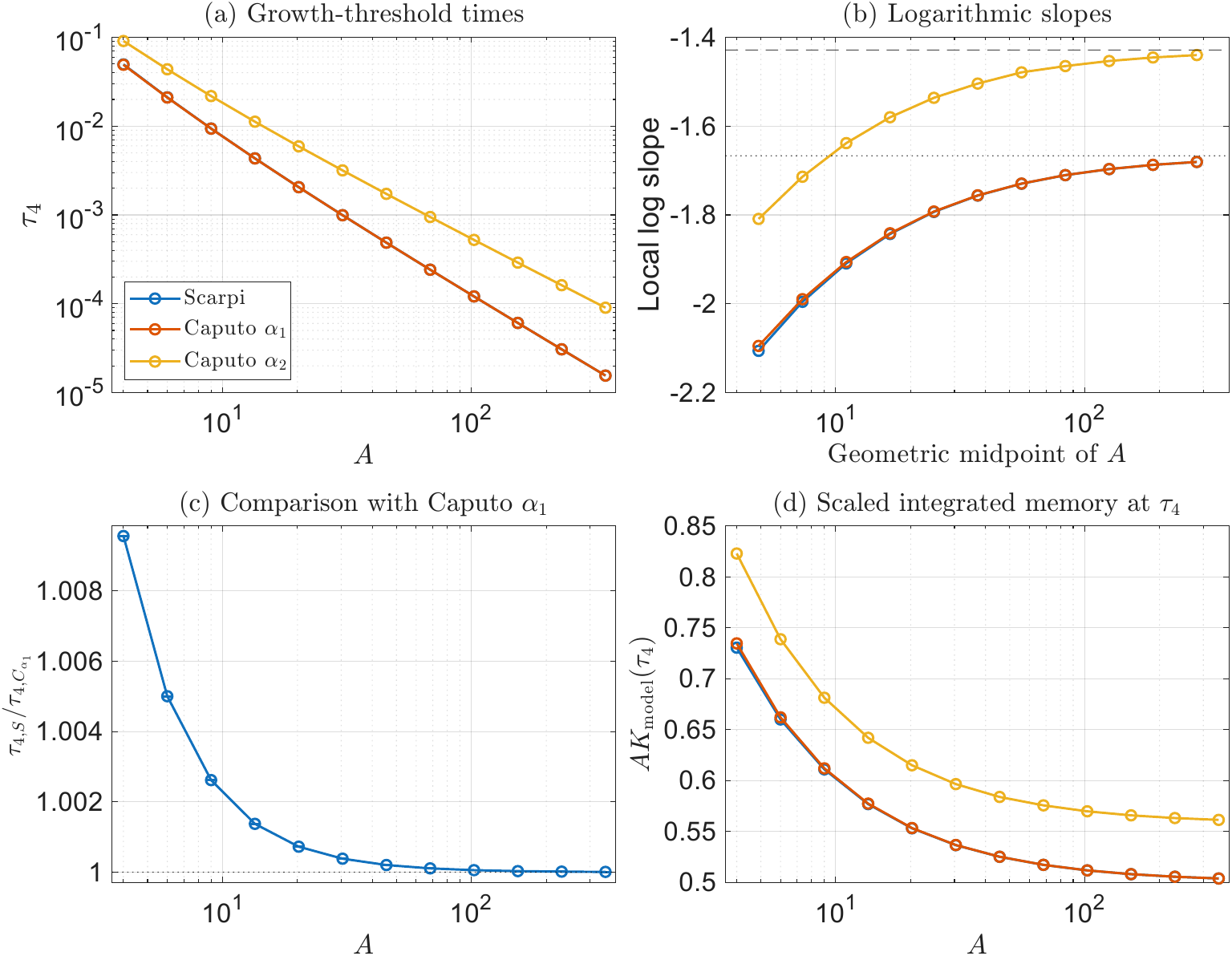}
\caption{Large-amplitude comparison for $u_0=A\e^{-x^2}$.
Panels (a,b) show $\tau_4$ and its adjacent-amplitude logarithmic slopes, with reference levels $-5/3$ and $-10/7$.
Panels (c,d) show the Scarpi-to-Caputo $\alpha_1$ threshold ratio and $AK_{\rm model}(\tau_4)$. The bars in (c) show the largest change in the ratio under time refinement, spatial refinement, or a change in the interpolation used to locate the threshold.}
\label{fig:endpoint-large}
\end{figure}

To compare the spatial solutions, let $V_A(x,\sigma)=A^{-1}U(x,A^{-1/\alpha_1}\sigma)$ for $0\le\sigma\le\sigma_{\rm end}$, where $\sigma_{\rm end}\approx0.1387963$.
The common interval ends before the peak doubles in any of the runs. Over 101 equally spaced times in this interval, the maximum relative $L^2$ difference between the rescaled Scarpi and Caputo $\alpha_1$ solutions decreases as $A$ increases, see Figure~\ref{fig:endpoint-large-trajectories}.

\begin{figure}[htbp]
\centering
\includegraphics[width=0.74\textwidth]{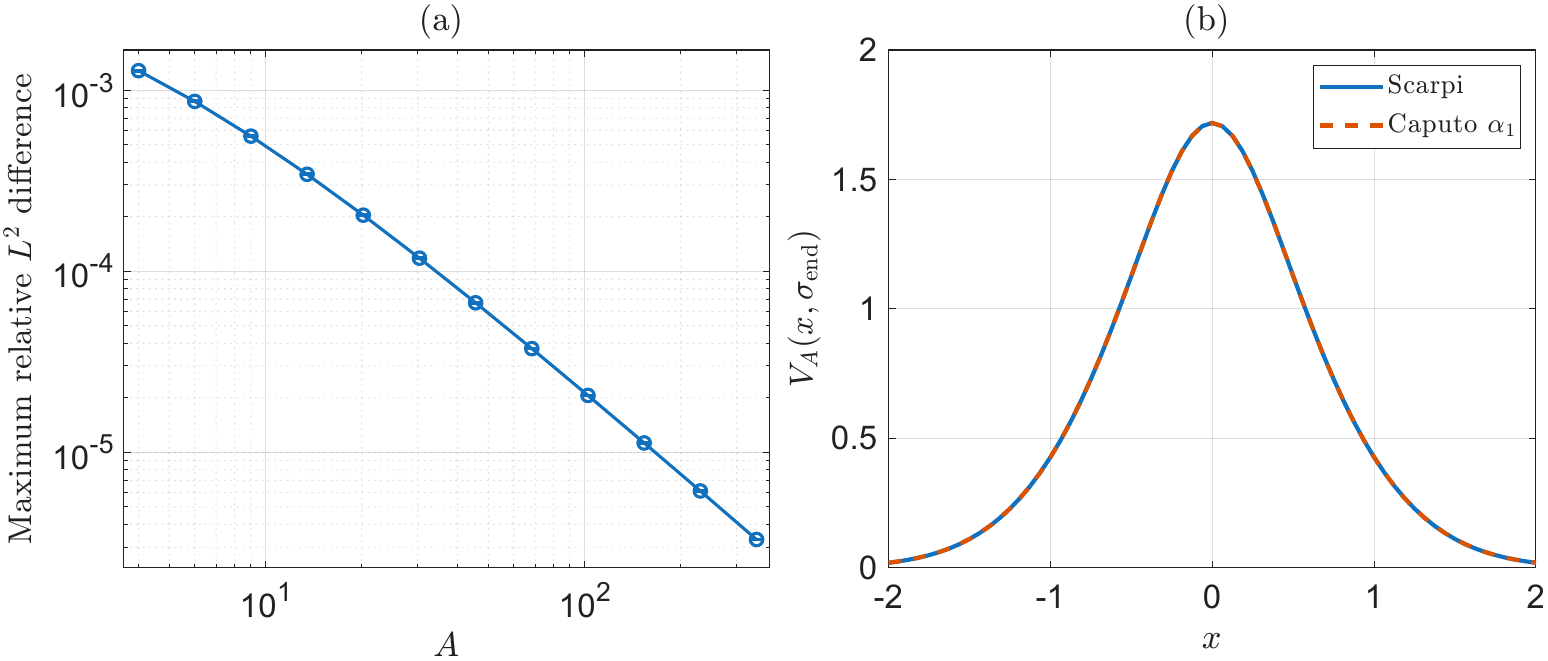}
\caption{Comparison of Scarpi and Caputo $\alpha_1$ solutions.
The left panel shows the maximum relative $L^2$ difference over the common time interval.
The error bars show the larger of the temporal and spatial refinement changes of the paired solution difference, each maximised over time and normalised by the Caputo reference norm.
The right panel shows the profiles at the largest amplitude and the final time of that interval.}
\label{fig:endpoint-large-trajectories}
\end{figure}

\subsection{Nonlinear rescaling for the long-time order}

The scaling is the constant-order Caputo scaling described in \cite{ZhangSun2015}.
We apply it with $\alpha=\alpha_2$ to keep the rescaled initial profile fixed and compare the Scarpi solutions with a single Caputo reference solution.

For $p=2$ and $b\in\{1,4,16,64,256,1024\}$, take
\[
 u_0(x)=4b^{-\alpha_2/(p-1)}
        \e^{-(x/b^{\alpha_2/2})^2},
 \qquad
 v_b(y,\sigma)=b^{\alpha_2/(p-1)}
              u(b^{\alpha_2/2}y,b\sigma).
\]
Then $v_b(y,0)=4\e^{-y^2}$.
Changing variables in the Volterra formulation gives the kernel $b^{1-\alpha_2}\psi(b\sigma)$ and therefore the symbol
\[
 \Psi_b(z)=b^{-\alpha_2}\Psi(z/b)
 \longrightarrow z^{-\alpha_2}
 \qquad(b\to\infty),
\]
for each $z>0$.
For Caputo $\alpha_2$, the transformed equation is independent of $b$, and its solution is denoted by $v_{C_{\alpha_2}}$.
For Caputo $\alpha_1$, the transformed symbol is $b^{\alpha_1-\alpha_2}z^{-\alpha_1}$.

For the computed solutions, define
\[
 e_b=\max_{j=0,\ldots,256}
 \frac{\norm{v_b(\cdot,\sigma_j)-v_{C_{\alpha_2}}(\cdot,\sigma_j)}_2}{\norm{v_{C_{\alpha_2}}(\cdot,\sigma_j)}_2},\qquad \sigma_j=j\sigma_*/256,
 \qquad \sigma_*\approx0.0414963.
\]
The same interval is used for every $b$.
Figure~\ref{fig:endpoint-terminal}(a) shows the relative $L^2$ difference at each $\sigma$, and Table~\ref{tab:endpoint-terminal} reports its maximum $e_b$.
The quantity $\tau_{\Lambda,b,\sigma}$ is the first threshold time at which $\norm{v_b(\cdot,\sigma)}_\infty$ reaches $4\Lambda$, and $\tau_{\Lambda,C_{\alpha_2}}$ is the corresponding Caputo threshold time in $\sigma$.
The threshold time in the original variable is $b\tau_{\Lambda,b,\sigma}$.

The maximum difference $e_b$ is smallest at the largest sampled $b$. It decreases overall, with a small rise from $b=64$ to $b=256$. For every $b$ the maximum is attained at the common final time $\sigma_*$, where the profiles are shown in Figure~\ref{fig:endpoint-terminal}(b).

At all three growth levels, the Scarpi solution reaches the threshold earlier than Caputo $\alpha_2$ for $b\le16$ and later for $b\ge64$ among the sampled values. The ratios decrease towards one from above at the largest sampled $b$.
The amplitude of the initial profile decreases while its width increases.
We next keep its width fixed to examine the small-amplitude regime of Theorem~\ref{thm:small-upper}.

\begin{table}[htbp]
\centering
\small
\caption{Comparison of rescaled Scarpi solutions with Caputo $\alpha_2$, reporting fourfold-growth time ratios and maximum relative $L^2$ differences $e_b$. The last column sums the maximum  temporal, spatial, and domain refinement for both solutions over the common time interval.}

\begin{tabular}{cccc}
\toprule
$b$ & $\tau_{4,b,\sigma}/\tau_{4,C_{\alpha_2}}$ &
$e_b$ & Sum of relative differences\\
\midrule
1 & 0.5445 & 0.5980 & 0.0008\\
4 & 0.6994 & 0.2265 & 0.0009\\
16 & 0.8975 & 0.0545 & 0.0010\\
64 & 1.0255 & 0.0142 & 0.0010\\
256 & 1.0347 & 0.0163 & 0.0009\\
1024 & 1.0175 & 0.0076 & 0.0005\\
\bottomrule
\end{tabular}
\label{tab:endpoint-terminal}
\end{table}

\begin{figure}[htb]
\centering
\includegraphics[width=0.85\textwidth]{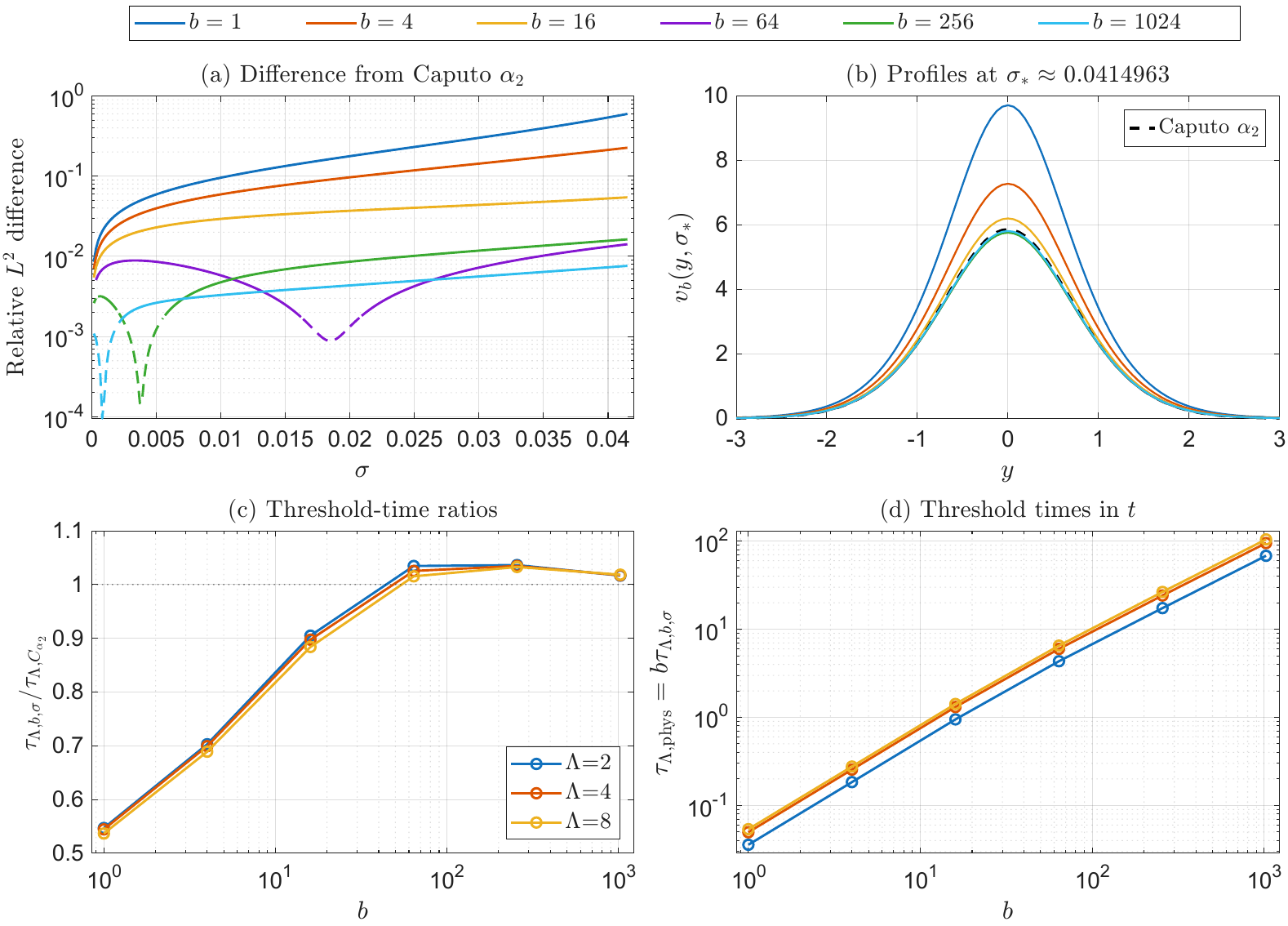}
\caption{For $p=2$ and initial profile $v_b(y,0)=4\e^{-y^2}$, panel~(a) shows the relative $L^2$ difference between each rescaled Scarpi solution and the Caputo solution of order $\alpha_2$.
Dashed segments mark times when the sum of the temporal, spatial, and domain refinement changes for both solutions exceeds $20\%$ of this difference.
Panel~(b) compares the spatial profiles at the common final time $\sigma_*$.
For each threshold $\norm{v}_\infty=4\Lambda$, panel~(c) gives the Scarpi-to-Caputo ratio of first hitting times in $\sigma$, and panel~(d) gives the corresponding Scarpi time in the original variable $t=b\sigma$.}
\label{fig:endpoint-terminal}
\end{figure}

\subsection{Fixed-profile amplitude and transition-rate effects}

\paragraph{Fixed-profile small amplitudes.}
We now vary the amplitude without changing the Gaussian width.
For $u_0=A_\varepsilon\e^{-x^2}$ with $p=3/2$ and $A_\varepsilon=2^{-j}$, $j=0,\ldots,4$, Theorem~\ref{thm:small-upper} motivates $A_\varepsilon^{2/3}K_{\rm model}(\tau_4)$, where the exponent $2/3$ equals $2(p-1)/(2-n(p-1))$ for $n=1$ and $p=3/2$.

As the initial amplitude decreases, the peak falls to a smaller fraction of its initial height before it recovers, see Figure~\ref{fig:endpoint-small}. At the centre of the Gaussian, $u_0(0)^p/|\Delta u_0(0)|=A_\varepsilon^{p-1}/2$. Thus the initial reaction term weakens relative to diffusion as $A_\varepsilon$ decreases. At the two larger amplitudes the Scarpi solution reaches fourfold growth earlier than Caputo $\alpha_2$ and at the three smaller amplitudes it reaches it later, see Table~\ref{tab:endpoint-fixed}.

\paragraph{Amplitude--rate effects.}
For $p=2$ and $u_0=A\e^{-x^2}$, Figure~\ref{fig:endpoint-rate} compares the three Scarpi transition rates with both Caputo models.
Increasing $c$ from the lowest to the highest rate delays fourfold growth at the two larger amplitudes and advances it at the smallest.
An amplitude-dependent order effect also occurs in the scalar Caputo equation with reaction $u^2$ \cite{FengLiLiuXu2018}.

The first two rows of Figure~\ref{fig:endpoint-rate} compare the normalised peak against $t$ and against $AK_{\rm model}(t)$.
The second coordinate uses each model's own integrated memory.
At the smallest amplitude, the Caputo solution of order $\alpha_1$ reaches fourfold growth later in $t$ than the solution of order $\alpha_2$, but at a smaller value of $AK_{\rm model}(t)$. The last row compares the shapes $U(x,\tau_4)/(4A)$ at each model's own threshold time for the same peak level $4A$.
At the largest amplitude the profiles nearly coincide. At the smallest amplitude the Caputo profile of order $\alpha_1$ is narrower than that of order $\alpha_2$, and the Scarpi profiles lie close to the latter.

\begin{table}[htbp]
\centering
\small
\caption{Fourfold growth times for fixed-width Gaussian initial data. The column $\tau_4$ gives the Scarpi time. Each ratio is taken with respect to the corresponding Caputo time at the same amplitude and $p$.}
\begin{tabular}{ccccc}
\toprule
$p$ & $\norm{u_0}_\infty$ & $\tau_4$ &
$\tau_4/\tau_{4,C_{\alpha_1}}$ & $\tau_4/\tau_{4,C_{\alpha_2}}$\\
\midrule
1.5 & 1 & 1.6621 & 1.0466 & 0.9185\\
1.5 & 0.5 & 3.4806 & 1.0156 & 0.9876\\
1.5 & 0.25 & 7.0714 & 0.9510 & 1.0260\\
1.5 & 0.125 & 14.0260 & 0.8668 & 1.0384\\
1.5 & 0.0625 & 27.4541 & 0.7783 & 1.0362\\
2 & 4 & 0.0497 & 1.0096 & 0.5444\\
2 & 1 & 1.4629 & 1.0480 & 0.8864\\
2 & 0.25 & 57.8504 & 0.6795 & 1.0251\\
\bottomrule
\end{tabular}
\label{tab:endpoint-fixed}
\end{table}

\begin{figure}[htbp]
\centering
\includegraphics[width=\textwidth]{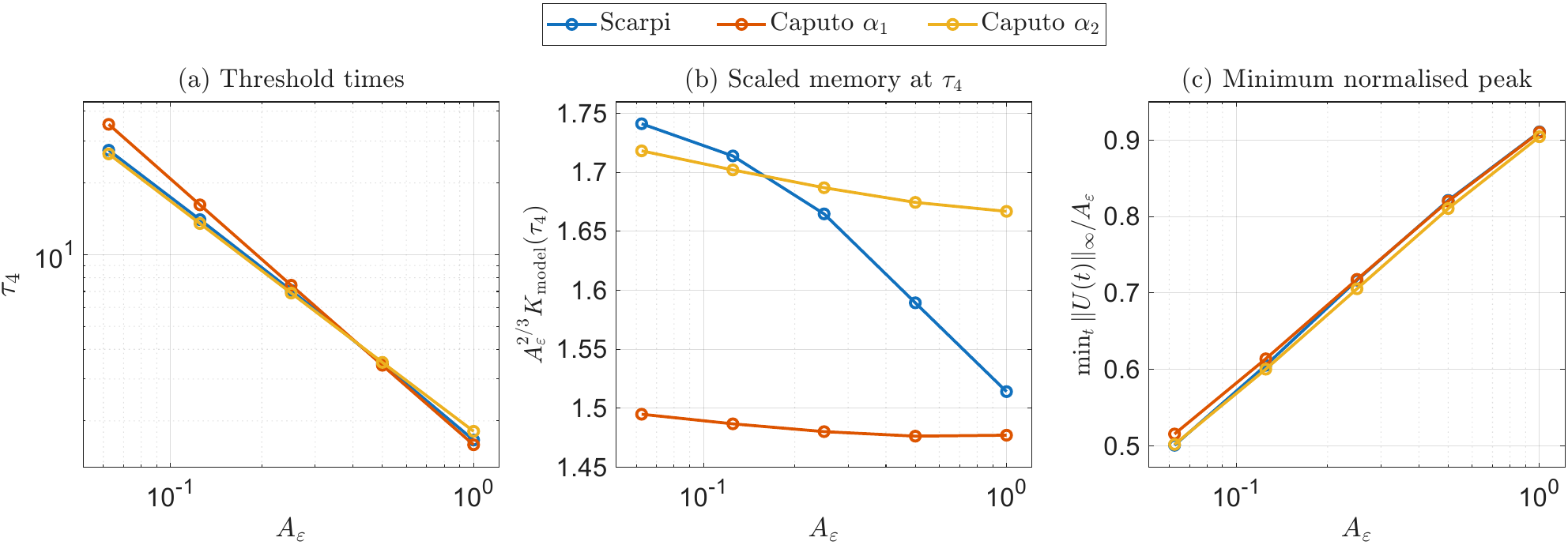}
\caption{Fixed-profile comparisons for $u_0=A_\varepsilon\e^{-x^2}$, $p=3/2$ and $A_\varepsilon=2^{-j}$, $j=0,\ldots,4$.
The panels show $\tau_4$, $A_\varepsilon^{2/3}K_{\rm model}(\tau_4)$, and the minimum of $\norm{U}_\infty/A_\varepsilon$ over saved time levels up to the first discrete crossing of $\Lambda=8$.}
\label{fig:endpoint-small}
\end{figure}

\begin{figure}[htb]
\centering
\includegraphics[width=0.97\textwidth]{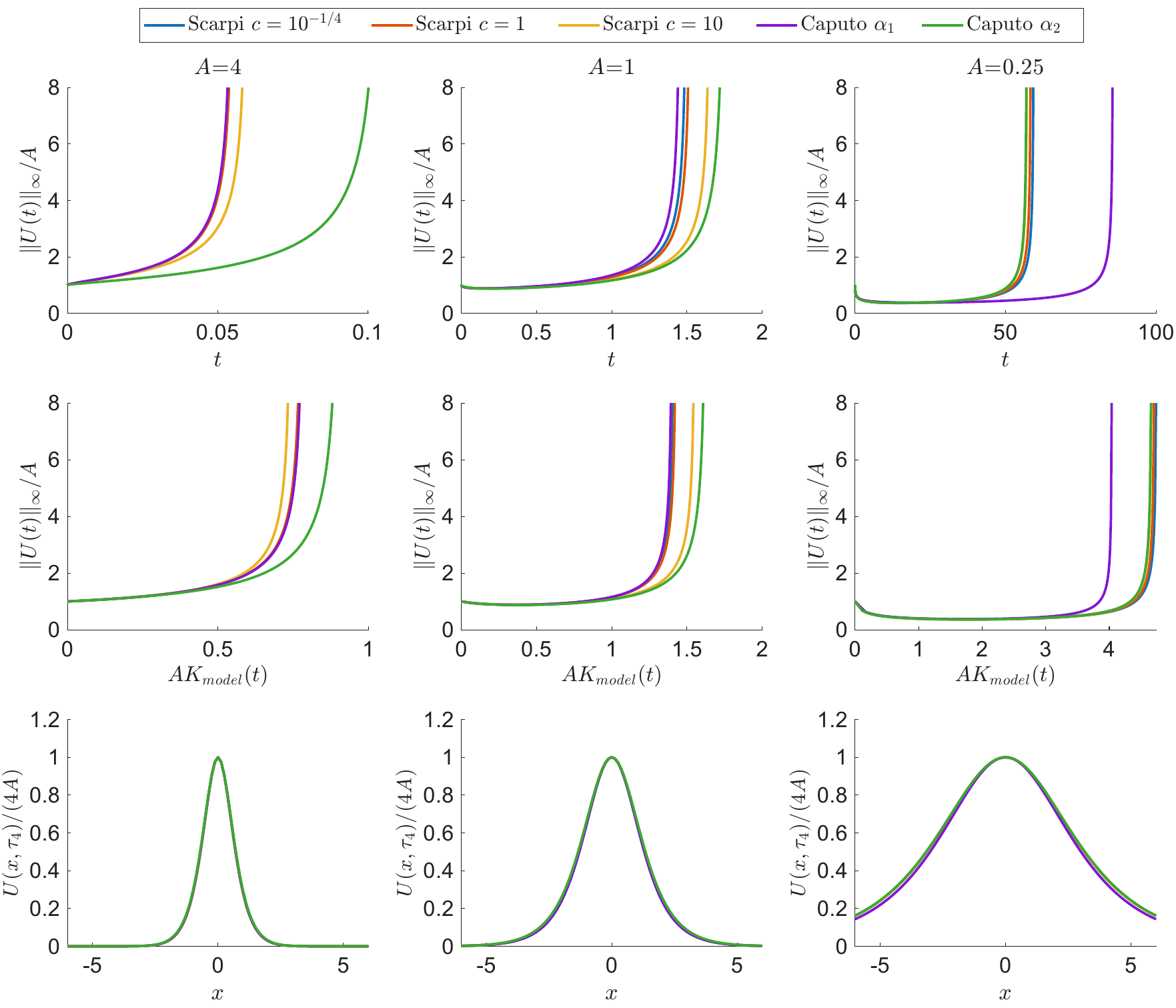}
\caption{Amplitude--rate comparisons for $u_0=A\e^{-x^2}$
and $p=2$.
Columns correspond to $A=4,1,0.25$. Within each column the Scarpi curves differ only in the transition rate.
Rows show normalised peaks against $t$, the same peaks against each model's own $AK_{\rm model}(t)$, and profiles $U(x,\tau_4)/(4A)$ at the corresponding threshold times on $|x|\le6$.}
\label{fig:endpoint-rate}
\end{figure}

\section{Conclusion}
\label{sec:conclusion}

For the abstract Scarpi–Volterra equation, we establish local well-posedness and the existence of a unique maximal mild solution. For the semilinear heat equation under the Bernstein assumption, we obtain a comparison principle, a mass identity for the solution, and an \(L^\infty\) blow-up alternative.
We give the explicit sufficient condition on $(\alpha_1,\alpha_2,c)$ for the Scarpi symbol of a nonconstant exponential transition to be a Bernstein function.

The Gaussian weighted-moment reduction establishes finite-time blow-up of all nontrivial nonnegative solutions for $1<p<1+2/n$, as well as blow-up for sufficiently large initial profiles for every $p>1$.
Lifespan bounds formulated in terms of $K^{-1}$ yield the matching large-amplitude scaling $T_A\asymp A^{-(p-1)/\alpha_1}$ as $A\to\infty$ and a subcritical small-amplitude upper bound governed by $\alpha_2$.

Numerically, the logarithmic slope $a_K$ of $K$ can lie outside the interval between the endpoint orders $\alpha_1$ and $\alpha_2$.
Over the computed range, large-amplitude thresholds and rescaled profiles approach those of the Caputo model of order $\alpha_1$.
Across all three tested growth levels, the Scarpi solution reaches the threshold earlier than the Caputo solution of order $\alpha_2$ at smaller rescaling factors and later at larger ones, with the threshold-time ratios approaching $1$ from above at the largest sampled factor.
For fixed-width Gaussian data, increasing $c$ from the lowest to the highest tested rate delays fourfold growth at the two larger amplitudes but advances it at the smallest, where the Caputo profile of order $\alpha_1$ is also narrower than that of order $\alpha_2$ at equal peak height.
These observations show that the endpoint lifespan scales coexist with nonmonotonic, amplitude-dependent growth dynamics.
Determining a sharp small-amplitude lifespan law and analysing global existence in the critical and supercritical regimes $p\ge 1+2/n$ remain open analytical problems for the Scarpi equation.

\appendix
\section{A sufficient condition for Bernstein admissibility}
\label{app:bernstein-b}

For the exponential transition, write
\[
 F(s)=\Psi(s)^{-1}=s^{\alpha_1+d_\alpha c/(s+c)},
 \qquad d_\alpha=\alpha_2-\alpha_1,
 \qquad L_c=\log c,
\]
where $0<\alpha_1,\alpha_2<1$ and $c>0$.
The following criterion verifies Assumption~\ref{ass:admissible}.

\begin{theorem}
\label{thm:BF-b-sufficient}
For $t>0$, define
\[
 W_c(t)=ct\,\e^{-ct}\bigl(\operatorname{Ei}(ct)-L_c\bigr),
 \qquad \mathcal V_c(t)=\int_0^t|W_c'(r)|\,dr,
 \qquad \mathcal M_c=\sup_{t>0}|tW_c'(t)|,
\]
where $\operatorname{Ei}$ is the real exponential integral, and set
$\mathcal V_c(\infty)=\int_0^\infty|W_c'(r)|\,dr$.
If
\begin{equation}
 |d_\alpha|\mathcal V_c(\infty)<\alpha_1
 \label{eq:BF-merged-smallness}
\end{equation}
and, for every $t>0$,
\begin{equation}
 d_\alpha\left(W_c(t)-\frac{tW_c'(t)}{\alpha_1}\right)
 +\frac{d_\alpha^2\mathcal M_c\mathcal V_c(t)}
 {\alpha_1\bigl(\alpha_1-|d_\alpha|\mathcal V_c(t)\bigr)}
 \le 1-\alpha_1,
 \label{eq:BF-merged-local}
\end{equation}
then $F$ is a Bernstein function.
\end{theorem}

\begin{proof}
The power-series representation of $\operatorname{Ei}(x)$ is given by
\[
 \operatorname{Ei}(x)=\gamma+\log x+
 \sum_{k=1}^{\infty}\frac{x^k}{k\,k!},\qquad x>0,
\]
where $\gamma$ is Euler's constant \cite{OlverEtAl2010}.
For large positive arguments, it has the asymptotic expansion
\[
 \operatorname{Ei}(x)=\frac{\e^x}{x}
 \left(1+\frac1x+\frac{2}{x^2}+O(x^{-3})\right),
 \qquad x\to\infty.
\]
Substitution into the definition of $W_c$ yields
\[
 \lim_{t\to0^+}W_c(t)=0,\qquad
 \lim_{t\to\infty}W_c(t)=1.
\]
Using these expansions in $W_c'(t)=c\left[1+(1-ct)\e^{-ct}
 \bigl(\operatorname{Ei}(ct)-L_c\bigr)\right]$ gives
\[
 W_c'(t)=O(1+|\log t|)\quad(t\to0^+)\qquad
 W_c'(t)=O(t^{-2})\quad(t\to\infty).
\]
Therefore, $W_c'\in L^1(0,\infty)$ and $t\mapsto tW_c'(t)$ is bounded.
Denote $L(s)=\log s/(s+c)$. Substituting the series for $\operatorname{Ei}(ct)$
into the Laplace integral and integrating termwise gives, for $s>0$,
\[
 \begin{aligned}
 \int_0^\infty\e^{-(s+c)t}\bigl(L_c-\operatorname{Ei}(ct)\bigr)\,dt
 &=\frac{\log(s+c)}{s+c}
 -\frac1{s+c}\sum_{k=1}^\infty\frac1k\left(\frac{c}{s+c}\right)^k=L(s).
 \end{aligned}
\]
The term $\e^{-(s+c)t}(\gamma+\log t)$ is absolutely integrable on $(0,\infty)$. Tonelli's theorem applies to the nonnegative series, whose integrated sum converges since $c/(s+c)<1$.
And because $\LT\{W_c\}(s)=cL'(s)$ and $W_c(0^+)=0$, the Laplace transform of $W_c'$ is
\[
 \LT\{W_c'\}(s)=s\LT\{W_c\}(s)=csL'(s).
\]
Since $|\LT\{W_c'\}(s)|\le\mathcal V_c(\infty)$, condition
\eqref{eq:BF-merged-smallness} implies
\[
 F'(s)=s^{\alpha_1-1}\e^{d_\alpha cL(s)}
       \bigl(\alpha_1+d_\alpha \LT\{W_c'\}(s)\bigr)>0.
\]
Consequently,
\[
 -(\log F')'(s)=\frac{1-\alpha_1}{s}-d_\alpha cL'(s)
 +\frac{(d_\alpha/\alpha_1)\LT\{t W_c'\}(s)}
 {1+(d_\alpha/\alpha_1)\LT\{W_c'\}(s)}.
\]
Define
\[
 \begin{aligned}
 \rho(t)={}&1-\alpha_1
 -d_\alpha\left(W_c(t)-\frac{tW_c'(t)}{\alpha_1}\right)\\
 &+\sum_{j=1}^{\infty}(-1)^j
 \left(\frac{d_\alpha}{\alpha_1}\right)^{j+1}
 \int_0^t(t-r)W_c'(t-r)\operatorname{Conv}^j(W_c')(r)\,dr,
 \end{aligned}
\]
where $\operatorname{Conv}^j(W_c')$ denotes the $j$-fold convolution of $W_c'$.
For every $j\ge1$,
\[
 \left|\int_0^t(t-r)W_c'(t-r)\operatorname{Conv}^j(W_c')(r)\,dr\right|
 \le\mathcal M_c\mathcal V_c(t)^j.
\] 
By \eqref{eq:BF-merged-smallness}, the series converges absolutely in
$L^\infty(0,\infty)$. By using the geometric series for
$[1+(d_\alpha/\alpha_1)\LT\{W_c'\}(s)]^{-1}$, it gives
$-(\log F')'=\LT\{\rho\}$.
Summing the convolution bounds, then
\[
 \rho(t)\ge1-\alpha_1
 -d_\alpha\left(W_c(t)-\frac{tW_c'(t)}{\alpha_1}\right)
 -\frac{d_\alpha^2\mathcal M_c\mathcal V_c(t)}
 {\alpha_1\bigl(\alpha_1-|d_\alpha|\mathcal V_c(t)\bigr)}\ge0
\]
by \eqref{eq:BF-merged-local}.
As $\rho$ is bounded and nonnegative, differentiation under the Laplace integral proves that $-(\log F')'$ is completely monotone.
The identity $F''=(\log F')'F'$, Leibniz's rule and induction then give $(-1)^m(F')^{(m)}\ge0$ for all $m\ge0$.
Hence $F'$ is completely monotone and $F$ is Bernstein.
\end{proof}

\begin{remark}
\label{rem:BF-merged-global}
Under \eqref{eq:BF-merged-smallness}, a simpler sufficient condition is
\begin{equation}
 \sup_{t>0}\left\{d_\alpha\left(W_c(t)-\frac{tW_c'(t)}{\alpha_1}\right)\right\}
 +\frac{d_\alpha^2\mathcal M_c\mathcal V_c(\infty)}
 {\alpha_1\bigl(\alpha_1-|d_\alpha|\mathcal V_c(\infty)\bigr)}
 \le1-\alpha_1.
 \label{eq:BF-merged-global}
\end{equation}
This follows from $\mathcal V_c(t)\le\mathcal V_c(\infty)$ and the monotonicity of $v/(\alpha_1-|d_\alpha|v)$ for
$0\le v\le\mathcal V_c(\infty)$.
For each fixed $\alpha_1\in(0,1)$ and $c>0$, the criterion holds for all
sufficiently small $|d_\alpha|$, since $W_c$ and $tW_c'$ are bounded.
\end{remark}

We now make the global criterion \eqref{eq:BF-merged-global} explicit at the three transition rates used in Section~\ref{sec:numerics}.

\begin{lemma}
\label{lem:BF-merged-bounds}
For $c\in\{10^{-1/4},1,10\}$, we define the following bounding constants
\begin{equation}
      \begin{aligned}
          &(U_c,J_c,V_c,M_c)=(8/5,1/2,12/5,110/109),\qquad &c&=10^{-1/4},\\
          &(U_c,J_c,V_c,M_c)=(3/2,1/2,7/3,6/5),\qquad &c&=1,\\
          &(U_c,J_c,V_c,M_c)=(9/7,1/4,97/35,15/11),\qquad &c&=10.
      \end{aligned}
 \label{eq:BF-b-bound-constants}
\end{equation}
Then, for every $t>0$,
\[
 W_c(t)<U_c,\qquad tW_c'(t)>-J_c,\qquad
 \mathcal V_c(\infty)<V_c,\qquad\mathcal M_c<M_c.
\]
\end{lemma}

\begin{proof}
Let $x=ct$ with $c$ fixed. Define
\[
 E(x)=\e^{-x}\bigl(\operatorname{Ei}(x)-L_c\bigr),
 \qquad w(x)=xE(x),\qquad R(x)=xw'(x).
\]
Then $W_c(t)=w(ct)$, $tW_c'(t)=R(ct)$ and $E'(x)=x^{-1}-E(x)$. The derivatives of $w$ and $R$ satisfy
\[
 w'(x)=1+(1-x)E(x),\qquad
 R'(x)=2-x+P(x)E(x),
 \qquad P(x)=x^2-3x+1.
\]
Thus, at a critical point of $w$, we have
\[
 w(x)=\frac{x}{x-1},\qquad
 w''(x)=-\frac1{x(x-1)}.
\]
Similarly, at a critical point of $R$, the equation $R'(x)=0$ gives
\[
 R(x)=-\frac{x}{P(x)},\qquad
 R''(x)=\frac{x+1}{xP(x)}.
\]

The endpoint expansions of $\operatorname{Ei}$ show that both $w'$ and $R'$ are negative near zero. When $x\to\infty$, their expansions are 
\[
 w'(x)=-x^{-2}+O(x^{-3}),\qquad
 R'(x)=x^{-2}+O(x^{-3}).
\]
Moreover, $w'(1)=1$. At every zero of $w'$ in $(0,1)$ and $(1,\infty)$, the above formula gives $w''>0$ and $w''<0$, respectively.
This imply that $w$ has exactly two critical points: a minimum at $\xi_1\in(0,1)$ and a maximum at $\xi_2>1$.

Let $r_\pm=(3\pm\sqrt5)/2$ be the roots of $P$. Then $R'(r_-)>0$ and $R'(r_+)<0$.
At the zeros of $R'$, the sign of $R''$ is positive,
negative and positive on $(0,r_-)$, $(r_-,r_+)$ and
$(r_+,\infty)$, respectively.
Together with the endpoint signs of $R'$, this proves that $R$ has two local minima $\eta_1\in(0,r_-)$ and $\eta_3\in(r_+,\infty)$ and one local maximum $\eta_2\in(r_-,r_+)$.

The endpoint limits are $w(0^+)=0,\qquad w(\infty)=1$ and $R(0^+)=R(\infty)=0$. The monotonicity intervals of $w$ therefore show that
$w(\xi_1)$ and $w(\xi_2)$ are its global minimum and maximum. So,
\[
 \mathcal V_c(\infty)
 =\int_0^\infty |w'(x)|\,dx
 =2\bigl(w(\xi_2)-w(\xi_1)\bigr)-1.
\]
Likewise, the three critical points of $R$ and its zero endpoint limits give
\[
 \mathcal M_c
 =\sup_{x>0}|R(x)|
 =\max_{j=1,2,3}|R(\eta_j)|.
\]

We first evaluate these bounds for $c=1$.
Elementary estimates based on the convergent series for $\operatorname{Ei}$, with remainder bounds, give
\[
 w'(1/7)>0,\qquad w'(3)>0,\qquad R'(1)>0,
 R'(3/2)<0,\qquad R'(5)<0.
\]
The first two inequalities, together with the sign changes of $w'$, imply $\xi_1<\frac17$ and $\xi_2>3$.
The function $x\mapsto x/(x-1)$ is strictly decreasing on each side of $1$.
Applying the critical-value formula for $w$ at $\xi_1$ and $\xi_2$ therefore yields
\[
 w(\xi_1)>-\frac16,\qquad w(\xi_2)<\frac32.
\]
Since these are the global extremal values of $w$, we obtain
\[
 -\frac16<w(x)<\frac32,\qquad
 \mathcal V_1(\infty)
 <2\left(\frac32+\frac16\right)-1=\frac73.
\]

The remaining three sign estimates  locate the maximum and the last local minimum of $R$: $1<\eta_2<\frac32$ and $\eta_3>5$.
Since
\(
 \frac{d}{dx}\left(-\frac{x}{P(x)}\right)
 =\frac{x^2-1}{P(x)^2}>0,
\)
$R(x)$ at the critical points gives
\[
 R(\eta_2)<-\frac{3}{2P(3/2)}=\frac65,\qquad
 R(\eta_3)>-\frac{5}{P(5)}=-\frac5{11}.
\]
At the first local minimum, $x<1$ and $R(\eta_1)>\eta_1-\frac{1-\eta_1}{6}>-\frac16$.
Both local minima of $R$ therefore exceed $-1/2$, and its local maximum is less than $6/5$, and this proves
$R(x)>-1/2$ for $x>0$ and $\mathcal M_1<6/5$.

For $c=10^{-1/4}$ and $c=10$, we apply the same argument with $L_c=-\frac14\log10$ and $L_c=\log10$, respectively.
The corresponding finite sign estimates, obtained from the same series bounds, bracket the critical points of
$w$ and $R$. Applying the critical-value formulas to these brackets, followed by the identities for
$\mathcal V_c(\infty)$ and $\mathcal M_c$, gives the remaining bounds in \eqref{eq:BF-b-bound-constants}.
\end{proof}
Combining Lemma~\ref{lem:BF-merged-bounds} with the global criterion \eqref{eq:BF-merged-global} gives the following sufficient parameter conditions.

\begin{corollary}
\label{cor:BF-b-certificate}
Let $c\in\{10^{-1/4},1,10\}$, $0<\alpha_1\le\alpha_2<1$ and
$0<\theta\le1$.
Using the constants in \eqref{eq:BF-b-bound-constants}, define
\begin{align}
 m_{0,c}&=\alpha_1-d_\alpha V_c,\label{eq:BF-merged-m0}\\
 m_{\theta,c}&=1-\theta\alpha_1
 -d_\alpha\left(\theta U_c+\frac{J_c}{\alpha_1}\right)
 -\frac{d_\alpha^2M_cV_c}{\alpha_1m_{0,c}},
 \qquad m_{0,c}>0.\label{eq:BF-merged-margin}
\end{align}
If $m_{0,c}>0$ and $m_{\theta,c}\ge0$, then $F^\theta$ is Bernstein.
\end{corollary}

\begin{proof}
Apply Remark~\ref{rem:BF-merged-global} to $F^\theta$, whose endpoint orders are $\theta\alpha_1,\theta\alpha_2$.
The smallness condition follows from $m_{0,c}>0$, and Lemma~\ref{lem:BF-merged-bounds} gives
\[
 \sup_{t>0}\left\{\theta d_\alpha W_c(t)
 -\frac{d_\alpha}{\alpha_1}tW_c'(t)\right\}
 \le d_\alpha\left(\theta U_c+\frac{J_c}{\alpha_1}\right).
\]
The factors $\theta$ cancel from the remainder bound, leaving
$d_\alpha^2M_cV_c/(\alpha_1m_{0,c})$.
Thus $m_{\theta,c}\ge0$ implies the global criterion.
\end{proof}

Table~\ref{tab:BF-b-certificates} lists explicit parameter triples verified by Corollary~\ref{cor:BF-b-certificate}.
In every row, $m_{0,c}>0$, and the positive margin $m_{1,c}$ or $m_{1/2,c}$ establishes the stated Bernstein property. We here denote \(\mathrm{BF}\) as the class of Bernstein functions.

\begin{table}[htbp]
\centering
\caption{Parameter triples with exact positive margins certifying
$F\in\mathrm{BF}$ or $\sqrt F\in\mathrm{BF}$.}
\label{tab:BF-b-certificates}
\begin{tabular}{ccccc}
\toprule
$(\alpha_1,\alpha_2)$ & $c$ & $m_{0,c}$
& Margin $m_{\theta,c}$ & Verified property\\
\midrule
$(0.6,0.7)$ & $10^{-1/4}$ & $9/25$
& $m_{1,c}=4369/98100$ & $F\in\mathrm{BF}$\\
$(0.6,0.7)$ & $1$ & $11/30$
& $m_{1,c}=13/330$ & $F\in\mathrm{BF}$\\
$(0.6,0.7)$ & $10$ & $113/350$
& $m_{1,c}=36199/1044120$ & $F\in\mathrm{BF}$\\
$(0.75,0.9)$ & $10^{-1/4}$ & $39/100$
& $m_{1/2,c}=61977/283400$ & $\sqrt F\in\mathrm{BF}$\\
$(0.75,0.9)$ & $1$ & $2/5$
& $m_{1/2,c}=81/400$ & $\sqrt F\in\mathrm{BF}$\\
$(0.75,0.9)$ & $10$ & $117/350$
& $m_{1/2,c}=2791/20020$ & $\sqrt F\in\mathrm{BF}$\\
\bottomrule
\end{tabular}
\end{table}

The rows with $(\alpha_1,\alpha_2)=(0.6,0.7)$ verify Assumption~\ref{ass:admissible}, while those with $(0.75,0.9)$ satisfy the hypothesis of Remark~\ref{rem:one-dimensional-positivity}.

For other order pairs $0<\alpha_1\le\alpha_2<1$ at the tabulated rates, the conditions of Corollary~\ref{cor:BF-b-certificate} can be tested with $\theta=1$ for $F$ or $\theta=1/2$ for $\sqrt{F}$.
Failure to satisfy these sufficient conditions leaves the Bernstein status of the corresponding function undetermined by this criterion.

\small
\bibliographystyle{ieeetr}
\bibliography{ref}

\end{document}